\documentclass[11pt]{article}
\usepackage[utf8]{inputenc}
\usepackage[T1]{fontenc}
\usepackage{lmodern}
\usepackage{xspace}
\usepackage{cancel}
\usepackage{amsfonts,amsmath,amssymb,amsthm,enumerate}
\usepackage[small]{titlesec}
\usepackage{comment}
\usepackage[dvipsnames]{xcolor}
\usepackage{url}
\usepackage[bookmarks,colorlinks,linkcolor=blue]{hyperref}
\usepackage{fancyhdr}
\usepackage[a4paper,body={160mm,235mm},centering]{geometry}
\newcommand{\rset}{\mathbb{R}}

\newcommand{\zset}{\mathbb{Z}}
\newcommand{\nset}{\mathbb{N}}
\newcommand{\nl}{\nolimits}
\newcommand{\ep}{\varepsilon}
\newcommand{\ind}{\mathbf{1}}
\newcommand{\fl}{\longrightarrow}
\newcommand{\e}{\mathbb{E}}

\newcommand{\p}{\mathbb{P}}
\newcommand{\lp}{\mathrm{L}}
\newcommand{\m}{\mathcal} 

\newcommand{\ot}{\overline{t}}
\newcommand{\ut}{\underline{t}}
\newcommand{\os}{\overline{s}}
\newcommand{\us}{\underline{s}}
\newcommand{\ur}{\underline{r}}

\newcommand{\lip}{\text{\rm Lip}}
\newcommand{\dotafter}[1]{#1.}
\titleformat{\section}[hang]
{\normalfont\large\bfseries}{\thesection.}{.5em}{\dotafter}[]
\titleformat{\subsection}[runin]
{\normalfont\bfseries}{\thesubsection.}{.4em}{}[.]
\titlespacing*{\subsection}{0pt}{3ex plus 1ex minus .2ex}{1em}
\titleformat{\paragraph}[runin]{\normalfont\bfseries}{\theparagraph.}{.4em}{}[.]
\titleformat{\subparagraph}[runin]{\normalfont\bfseries}{\theparagraph.}{.4em}{}[.]
\theoremstyle{plain}
\newtheorem{thm}{Theorem}
\newtheorem{lemme}[thm]{Lemma}
\newtheorem{prop}[thm]{Proposition}
\newtheorem{cor}[thm]{Corollary}
\theoremstyle{definition}

\newtheorem{hyp}{Assumption}

\theoremstyle{remark}
\newtheorem{rem}[thm]{Remark}

\title{\bf Donsker-Type Theorem for BSDEs: \\[1ex] Rate of Convergence}
\author{
Philippe Briand 
\thanks{In memory of Jean Mémin from whom I learned lots of mathematics.}
\thanks{Many thanks to Pierre Baras for very fruitful discussions about the heat equation.}
\and 
Christel Geiss 
\and 
Stefan Geiss 
\and 
Céline Labart
\footnotemark[2]
}

\begin{document}
\maketitle
\begin{abstract}
  In this paper, we study in the Markovian case the rate of convergence in the
  Wasserstein distance of an approximation of the solution to a BSDE given by
  a BSDE which is driven by a scaled random walk as introduced in Briand,
  Delyon and M{\'e}min (Electron. Comm. Pro-\- bab. \textbf{6} (2001), 1--14).
\end{abstract}
\section{Introduction} % (fold)
%\label{sec:introduction}
%
%
In this paper, we are concerned with the discretization of solutions to BSDEs of the form
\begin{equation*}
	Y_t = G(B) + \int_t^T f(B_s,Y_s,Z_s)\, ds - \int_t^T Z_s\, dB_s, \quad 0\leq t\leq T,
\end{equation*}
where $B$ is a standard Brownian motion. These equations have been introduced
by Jean-Michel Bismut for linear generators in \cite{B73} and by Étienne
Pardoux and Shige Peng for Lipschitz generators in \cite{PP90}.

In one of the first studies on this topic, in the case where the generator $f$
may depend on $z$ as well, Philippe Briand, Bernard Delyon and Jean
Mémin~\cite{BDM01} proposed an approximation based on Donsker's theorem. They
showed that the solution $(Y,Z)$ to the previous BSDE can be approximated by
the solution $(Y^n,Z^n)$ to the BSDE
\begin{equation*}
  Y^n_t = G(B^n) + \int_{]t,T]} f(B^n_{s-},Y^n_{s-},Z^n_{s})\, d\langle B^n\rangle_s + \int_{]t,T]} Z^{n}_s\, dB^n_s, \quad 0\leq t\leq T,
\end{equation*}
where $B^n$ is the scaled random walk
\begin{equation*}
	B^n_t = \sqrt{T/n}\, \sum_{k=1}^{[nt/T]} \xi_k, \quad 0\leq t\leq T,
\end{equation*}
and $(\xi_k)_{k\geq 1}$ is an i.i.d. sequence of symmetric Bernoulli random
variables. They proved, in full generality, meaning that $G(B)$ is only
required to be a square integrable random variable, that $(Y^n,Z^n)$ converges
to $(Y,Z)$. However, the question of the rate of convergence was left
open. Right now it seems to be hopeless to get a result in this direction for
such a general path-dependent terminal condition $G(B).$ But in the Markovian
case, meaning that $G(B)=g(B_T)$, this problem seems to be tractable, in
particular due to the PDE structure behind. Indeed, $(Y,Z)$ is related to the
semilinear heat equation
\begin{equation}\label{eq:mainPDE}
	\partial_t u(t,x) + \frac{1}{2} \Delta u(t,x) + f(t,x,u(t,x),\nabla u(t,x))=0, \quad (t,x)\in[0,T[\times\rset, \qquad u(T,\cdot)=g,	
\end{equation}
where, under certain regularity conditions,  we can choose 
\begin{equation*}
	Y_s = u(s,B_s) \,\, \text{ and } \,\,   Z_s = \nabla u(s,B_s).
\end{equation*} 
In the case where $B^n$ is the discretized Brownian motion, the link to PDEs
was exploited in \cite{Zha04,BT04} to get the rate of convergence, in the
Markovian case, of the classical scheme for BSDEs.  The convergence of
this scheme was already proved in  \cite[Proposition 13]{BDM02} for a general terminal
condition  and a generator that is Lipschitz in its spatial coordinates  but without any rate of convergence.

Even though the link with PDEs was pointed out in \cite{BDM01}, the rate of
convergence of the approximation of BSDEs given by scaled random walks was
completely open.  In two recent papers, Christel Geiss, Céline Labart and
Antti Luoto~\cite{GLL1,GLL2} give a first answer to this question. They showed
that the error between $(Y^n,Z^n)$ and $(Y,Z)$ is of order $n^{-\ep/4}$ when
$g$ is assumed to be $\ep$-Hölder continuous and $f(t,\cdot)$ Lipschitz
continuous. One of the main arguments in these papers consists in constructing
the random walk from the Brownian motion $B$ using the Skorohod embedding
(see~\cite{Wal03}) together with generalizations of the pioneering work of Jin
Ma and Jianfeng Zhang~\cite{MZ02} on representation theorems for BSDEs. This
approach allows to work with convergence in the $\lp^2$-sense even if the
problem naturally arises in the weak sense. The drawback is that the rate of
convergence $n^{-\ep/4}$ obtained in these papers is not optimal as one can
expect $n^{-\ep/2}$.
\medskip

The objective of our study is to confirm this expected rate $n^{-\ep/2}$. This improvement was 
possible by using a weak limit approach, where the error is considered in the Wasserstein
distance.  Our starting point is a result of Emmanuel Rio~\cite{Rio09} who
proved that, when $T=1$, for all $r\geq 1$, there exists a constant $C_r$ such
that, for all $n\geq 1$, $W_r(B^n_1,G) \leq C_r\, n^{-1/2}$, where $W_r$ is
the $\lp^r$-Wasserstein distance and $G$ a standard normal random variable
(see Section~\ref{sec:scaled_random_walk_and_wasserstein_distance}). Firstly,
we generalize this result to cover the case where $f\equiv 0$ which
corresponds to the heat equation. Then, using the associated PDE, in
particular representation formulas in the spirit of~\cite{MZ02}, we are able
to prove  that
\begin{equation*}
	W_r(Y^n_t,Y_t)\leq C_r\, n^{-(\alpha \wedge \frac{\ep}{2})} \quad \text{ and } \quad W_r(Z^n_t,Z_t)\leq \frac{C_r}{\sqrt{T-t}}\, n^{-(\alpha \wedge \frac{\ep}{2})}
\end{equation*}
for  $t \in [0,T]$ and   $t \in [0,T[$, respectively,
when $g$ and  $f(t,\cdot,y,z)$ are $\ep$-Hölder continuous
and $f(\cdot,x,y,z)$ is $\alpha$-Hölder continuous. We refer to Theorem~\ref{en:mainBSDE} in
Section~\ref{sec:main_results} for the precise statement.
One of the main difficulties in the proof concerned various gradient estimates in order to obtain the estimate for $W_r(Z^n_t,Z_t)$.

For $\ep=1$ and $\alpha \ge 1/2$ we obtain the rate 
$n^{-\frac{1}{2}}$ which is the same rate as obtained from Rio
 for the Random walk approximation of a Gaussian random variable in the Wasserstein distance as mentioned above.

%
% section introduction (end)

\section{Notation} % (fold)
% \label{sec:notations}
%
In all the sequel, $T>0$ is a fixed positive real number. We work on a
complete probability space $(\Omega,\m F,\p)$ carrying a standard real
Brownian motion $\{B_t\}_{0\leq t\leq T},$ and $\{\m F_t\}_{0\leq t\leq T}$
stands for the augmented filtration of $B$ which is right continuous and
complete.

We consider the following BSDE
\begin{equation} \label{BSDE}
	Y_t = g(B_T) + \int_t^T f(s,B_s,Y_s,Z_s)\, ds - \int_t^T Z_s\, dB_s, \quad 0\leq t\leq T.
\end{equation}
Throughout this article, we will assume for the function $g$ defining the
terminal condition and the generator $f$ the following:

\begin{hyp}\label{H:A1}
	\renewcommand{\labelenumi}{(\roman{enumi})}
	There exist $0<\ep\leq 1$ and $0<\alpha\leq 1$ such that it holds:
	\begin{enumerate}
    \item The function $g:\rset\fl\rset$ is $\ep$-H\"older continuous: for all
      $(x,x')\in\rset^2$ one has
		\begin{equation*}
			 \left|g(x)-g\left(x'\right)\right|\leq \|g\|_\ep\, \left|x-x'\right|^{\ep}.
		\end{equation*}
      \item The function $f:[0,T]\times\rset\times\rset\times\rset \fl\rset$
        is $\alpha$-H\"older continuous in time, $\ep$-H\"older continuous in
        space and Lipschitz continuous with respect to $(y,z)$: for all
        $(t,x,y,z)$ and $(t',x',y',z')$ in
        $[0,T]\times\rset\times\rset\times\rset$ one has
		\begin{multline} \label{eq:prop_f}
			\left|f(t,x,y,z)-f\left(t',x',y',z'\right)\right| \\
			 \leq \|f_t\|_\alpha \, \left|t-t'\right|^\alpha + \|f_x\|_\ep\, \left|x-x'\right|^\ep + \|f_y\|_\lip\, \left|y-y'\right|+\|f_z\|_\lip\,\left|z-z'\right|.
		\end{multline}
	\end{enumerate}
\end{hyp}
Most of the time, we do not need to distinguish between $\|f_y\|_\lip$ and
$\|f_z\|_\lip$ and we let
$\|f\|_\lip := \max\left(\|f_y\|_\lip,\|f_z\|_\lip\right)$. 
\medskip

 {\bf Convention:} Later the phrase that {\it a constant $C>0$ depends on $(T,\varepsilon,f,g)$} stands for the
fact that $C$ can be expressed in terms of 
$(T,\varepsilon,\|f_x\|_\varepsilon,\|f_y\|_{\lip},\|f_z\|_\lip,K_f,\|g\|_\varepsilon,g(0))$
where
\begin{align*}
K_f:=\sup_{t \in[0,T]} |f(t,0,0,0)|.
\end{align*}
Similarly, a dependence on  $(T,\alpha,\varepsilon,f,g)$ means an additional dependence on $(\alpha,\|f_t\|_\alpha)$.

\medskip

From \cite[Theorem 4.2]{BDHPS} it is known that under \ref{H:A1}, the BSDE
\eqref{BSDE} has a unique $\lp^p$-solution $(Y,Z)$ for any
$p\in]1,\infty [.$ So
for $(t,x) \in [0,T[\times \rset$ we let  $\left(Y^{t,x}_s,Z^{t,x}_s\right)_{s \in [t,T]}$  be the square integrable solution to the BSDE
\begin{equation}\label{eq:mainBSDE}
	Y^{t,x}_s = g\left(B^{t,x}_T\right) + \int_s^T f\left(r,B^{t,x}_r,Y^{t,x}_r,Z^{t,x}_r\right) dr - \int_s^T Z^{t,x}_r\, dB_r, \quad t\leq s\leq T,
\end{equation} 
where $B^{t,x}_r := x+B_r-B_t$, 
and set, as usual, for $x\in\rset$, $u(T,x):=g(x)$, and, for $(t,x)\in[0,T[\times \rset$,
\begin{align*}
u(t,x):=Y^{t,x}_t = \e\left[g (B^{t,x}_T) +  \int_t^T f\left(r,B^{t,x}_r,Y^{t,x}_r,Z^{t,x}_r \right) dr\right].
\end{align*}
It is well known that the function $u$ is continuous on $[0,T]\times\rset$  (see also Lemma \ref{en:regnabus} below) and  under Lipschitz assumptions in $(x,y,z)$ and for $\alpha \ge \tfrac{1}{2}$ it  is the viscosity solution to~\eqref{eq:mainPDE},  see \cite[Theorem 5.5.8] {Zha17}.~Moreover, in this Markovian setting, for $(t,x)\in[0,T]\times\rset$,  we have $Y^{t,x}_s =  u(s, B^{t,x}_s)$   a.s. for all $s \in [t,T]$. In~\cite[Theorem 3.2]{Zha05}, for a generator which is Lipschitz continuous in all space variables and a measurable $g$ with polynomial growth, J.~Zhang proved that $u$ belongs to $\m C^{0,1}\left([0,T[\times\rset\right)$ and that $Z^{t,x}_s = \nabla u(s,B^{t,x}_s)$ a.e. on $[t,T[\times \Omega$. Moreover, the following representation holds
\begin{equation*}
    \nabla u(t,x)=\e\left[ g(B^{t,x}_T)\frac{B_T-B_t}{T-t}+ \int_t^T
      f\left(r,B^{t,x}_r,Y^{t,x}_r,Z^{t,x}_r \right)\frac{B_r-B_t}{r-t}\, dr\right], \quad (t,x)\in [0,T[\times\rset.
\end{equation*}
If  $F$ is the function given by
\begin{equation}\label{eq:defF}
	F(s,x):=f(s,x,u(s,x),\nabla u(s,x)) \quad \text{for } \,\, (s,x) \in [0,T[\times \rset,
\end{equation}
we thus have 
\begin{equation}\label{eq: u}
u(t,x) =  \e \left[g (B^{t,x}_T) +  \int_t^T F(r,B^{t,x}_r)\, dr\right], \quad (t,x)\in[0,T[\times\rset,
\end{equation}
together with
  \begin{equation}\label{eq:nabla-u}
    \nabla u(t,x)=\e\left[ g(B^{t,x}_T)\frac{B_T-B_t}{T-t} + \int_t^T F(r,B^{t,x}_r)\frac{B_r-B_t}{r-t}\, dr\right], \quad (t,x)\in[0,T[\times\rset.
  \end{equation}
These formulas play an important role in the sequel.

In Section \ref{sec:regularity_results_on_u_u_n_nabla_u_and_delta_n} and in the appendix, we extend these results to the case
where $f(t,\cdot,y,z)$ is $\ep$-H\"older continuous and  make the regularity of $u$ and $\nabla u$  precise.
\medskip

As mentioned before, we are concerned with the approximation of the solution
$\left(Y^{t,x},Z^{t,x}\right)$ to~\eqref{eq:mainBSDE} by a solution to the
BSDE driven by a scaled random walk. To do this, let us consider, on some
probability space, not necessarily $\left(\Omega,\m F,\p\right)$, an
i.i.d. sequence $(\xi_k)_{k\geq 1}$ of symmetric Bernoulli random
variables. For $n\in\nset^*:=\{1,2,3,...\}$ we set $h:=T/n$ and we consider
the scaled random walk
\begin{equation*}
	B^n_t := \sqrt{h}\, \sum_{k=1}^{[t/h]} \xi_k, \quad 0\leq t\leq T,
\end{equation*}
where $[x] := \max\{r\in\zset : r\leq x\}$ for any real number $x.$ As we did
for the Brownian motion, for $x\in\rset$ and $0\leq t\leq s\leq T$ we
put $$B^{n,t,x}_s:=x+B^n_s-B^n_t.$$

Let us introduce some further notation. We denote the ceiling function by
$\lceil x \rceil:=\min \{r\in\zset : r \geq x\}$ for $x\in\rset.$ Moreover, we
set
\begin{equation*}
   n_t:=[t/h], \quad  \ut:=h[t/h]=hn_t  \quad \text{ and } \quad  \ot:=h\lceil t/h \rceil, \quad t \in [0,T].
\end{equation*}
For $n\in\nset^*$ let us consider the following BSDE driven by $B^n$:
\begin{equation*}
	Y^n_t=g(B^n_T)+\int_{]t,T]} f(s,B^n_{s^-},Y^n_{s^-},Z^n_s)\, d\langle B^n\rangle_s - \int_{]t,T]} Z^n_s\, dB^n_s, \quad t \in [0,T].
\end{equation*}
It was shown in \cite{BDM01} that, as soon as $h\, \|f\|_\lip<1$, this BSDE
has a unique square integrable solution $(Y^n,Z^n)$, $Y^n$ being adapted and
$Z^n$ being predictable with respect to the filtration generated by $B^n$. By
construction, $Y^n$ is a piecewise constant càdlàg process with
$Y^n_t=Y^n_{\ut}$. The process $Z^n$ is defined as an element of
$\lp^2(\Omega\times[0,T],d\p \otimes d\langle B^n\rangle),$ where we start
with a $Z^n$ defined only on the points $\{kh: k=1,\ldots,n\}$ and extend it
to $]0,T]$ as a càglàd process $(Z^n_t)_{t\in ]0,T]}$ by setting
$Z^n_t=Z^n_{\ot}$. The previous BSDE is actually a discrete BSDE that can be
solved by hand since, for $k=0,\cdots,n-1$, we have
\begin{align*}
  Y^n_{kh} = Y^n_{(k+1)h}+h\,f\left((k+1)h,B^{n}_{kh},Y^n_{kh},Z^n_{(k+1)h}\right)-\sqrt{h}\,Z^n_{(k+1)h}\xi_{k+1},\quad Y^n_{nh}=g(B^n_T).
\end{align*}
Thus, if $Y^{n}_{(k+1)h}$ is given,
\begin{align}
	Z^{n}_{(k+1)h} & = h^{-1/2}\,\e\left(Y^{n}_{(k+1)h}\xi_{k+1}\,|\,\m F^n_{kh}\right),     \label{z-eqn}   \\
	Y^n_{kh} & = Y^n_{(k+1)h}+h\,f\left((k+1)h,B^{n}_{kh},Y^n_{kh},Z^n_{(k+1)h}\right)-\sqrt{h}\,Z^n_{(k+1)h}\xi_{k+1}  \notag  \\
	& = \e\left(Y^{n}_{(k+1)h}\,|\,\m F^n_{kh}\right) + h\,f\left((k+1)h,B^{n}_{kh},Y^n_{kh},Z^n_{(k+1)h}\right), \label{y-eqn}
\end{align}
where the last equality follows by taking the conditional equation
w.r.t.~$\m F^n_{kh}$ of the second line.

Since we are in a Markovian setting, there is also an analog of the
Feynman-Kac formula. If $u$ is a given function we set
\begin{equation*}
	D^n_{+}u(x):=\frac{1}{2}\left(u(x+\sqrt{h})+u(x-\sqrt{h})\right),\qquad D^n_{-}u(x):=\frac{1}{2}\left(u(x+\sqrt{h})-u(x-\sqrt{h})\right),
\end{equation*}
and  
\begin{equation} \label{nabla}
\nabla^n u(x):=h^{-1/2}\,D^n_{-}u(x).
\end{equation}
\bigskip 
\begin{rem}\label{rem3}
  From the definition of $D^n_+$ and $D^n_-$, we get that if $u$ is $\ep$-Hölder, $D^n_+
  u$ and $D^n_-u$ are also $\ep$-Hölder with constant $\|u\|_{\ep}$.
\end{rem}
\bigskip 
Let $U^n$ be the solution to the finite difference equation, where for $x\in\rset$ and $k=0,\ldots, n-1$ we require 
\begin{align}\label{eq:mainPDEn}
\left \{ \begin{array}{l}	 U^n(kh,x)=D^n_+
    U^n((k+1)h,x)+hf((k+1)h,x,U^n(kh,x),h^{-1/2}D^n_-U^n((k+1)h,x)),  \\
  U^n(nh,x)=g(x).
\end{array} \right .  
\end{align}
Then, we obtain from  \eqref{z-eqn} and \eqref{y-eqn} (cf \cite[Proposition 5.1]{BDM01})  that 
\begin{align*}
 Y^n_{kh} & = U^n\left(kh,\sqrt{h}\,\sum_{i={ 1}}^{k} \xi_i\right), \quad k=0,\ldots,n, \\
 Z^n_{kh} & = \nabla^n U^n\left(kh, \sqrt{h}\,\sum_{i={ 1}}^{k-1} \xi_i\right), \quad k=1,\ldots,n.
\end{align*}
These formulas rewrite in continuous time  to
\begin{align*}
	Y^n_t=Y^n_{\ut}= U^n(\ut,B^n_t), \quad t\in [0,T], \quad  \text{and} \quad Z^n_t=Z^n_{\ot}= \nabla^n U^n\left(\ot,B^n_{t-}\right), \quad t\in  ]0,T].
\end{align*}
If we set, for $0\leq t\leq T$ and $x\in\rset$, $U^n(t,x):= U^n(\ut,x)$, we have $Y^n_t=U^n(t,B^n_t)$.

More generally, for $0\leq t <T$, we define $(Y^{n,t,x},Z^{n,t,x})$ as the solution $Y^{n,\ut,x}=(Y^{n,\ut,x}_s)_{s \in [\ut, T]}$ and 
$Z^{n,\ut,x}=(Z^{n,\ut,x}_s)_{s \in ]\ut, T]}$ to the BSDE 
\begin{equation}\label{eq:mainBSDEn}
	Y^{n,\ut,x}_s=g(B^{n,t,x}_T)+\int_{]s,T]} f(r,B^{n,t,x}_{r^-},Y^{n,\ut,x}_{r^-},Z^{n,\ut,x}_r) d\langle B^n\rangle_r - \int_{]s,T]} Z^{n,\ut,x}_r dB^n_r, 
	\quad  s \in [\ut, T].
\end{equation}
We set $Y^{n,T,x}_T=g(x)$. Then, 
\begin{align}\label{eq:RY}
 	Y^{n,t,x}_s=Y^{n,\ut,x}_{\us}=U^n(s,B^{n,t,x}_s), \quad 0\leq t\leq s\leq T.
\end{align}
 Let us observe that $Z^{n,\ut,x}$ is first defined at the points $t=kh$, $k=n_t+1,\ldots,n$. As before we let $Z^{n,\ut,x}_s := Z^{n,\ut,x}_{\os}$ for $s \in \,]\ut,T]$.
 We have
 \begin{align*}
  Z^{n,t,x}_s=Z^{n,\ut,x}_{\os}=\nabla^n U^n(\os,B^{n,t,x}_{s^-}) \quad \text{ for } s \in \,]\ut,T].
\end{align*}

In particular,
\begin{equation}\label{eq:RZZ}
	 Z^{n,t,x}_s = \nabla^n U^n(\us+h,B^{n,t,x}_s) \quad \text{whenever} \,\, s\in \, ]\ut,T] \backslash \{kh: k=n_t+1,\ldots,n\}.
\end{equation}
Of course, we have
\begin{equation*}
	U^n(t,x)=Y^{n,t,x}_t=Y^{n,\ut,x}_{\ut} \quad \text{ for }  t\in [0,T].
\end{equation*}
Similarly, we define, for $(t,x)\in [0,T[\times\rset$,
\begin{equation} \label{Delta-n}
	\Delta^n(t,x) := \nabla^n U^n(\ut+h,x) = Z^{n,\ut,x}_{\ut+h}.
\end{equation}
With this notation, \eqref{eq:RZZ} rewrites as
\begin{equation}\label{eq:RZ}
	 Z^{n,t,x}_s =\Delta^n(s,B^{n,t,x}_s) \quad \text{whenever} \,\,  s\in \, ]\ut,T] \backslash \{kh: k=n_t+1,\ldots,n\}.
\end{equation}
It follows that
\begin{align*} 
  U^n(t,x)& = \e\left[g\left(B^{n,t,x}_T\right) + h \, \sum_{k=n_t+1}^n f\left(kh,B^{n,t,x}_{(k-1)h},Y^{n,t,x}_{(k-1)h},Z^{n,t,x}_{kh}\right)\right]  \\
			& = \e\left[ g\left(B^{n,t,x}_T\right) + \int_{\ut}^T
            f\left(\os,B^{n,t,x}_{s},Y^{n,t,x}_{s},Z^{n,t,x}_{s}\right) ds \right],
\end{align*}
which rewrites, taking into account \eqref{eq:RY} and \eqref{eq:RZ}, to
\begin{align}  \label{U-n}
	U^n(t,x) & = \e\left[g\left(B^{n,t,x}_T\right)+ \int_{\ut}^T f\left(\os,B^{n,t,x}_s, U^n(s,B^{n,t,x}_s),\Delta^n (s,B^{n,t,x}_{s})\right) ds\right] \notag  \\
	& = \e\left[g\left(B^{n,t,x}_T\right)+ \int_{\ut}^T f\left(\os,\Theta^{n,t,x}_s\right) ds\right],
\end{align}
where
\begin{align}  \label{theta-n}
\Theta^{n,t,x}_s := \left(B^{n,t,x}_s,U^{n}(s,B^{n,t,x}_s),\Delta^n(s,B^{n,t,x}_s)\right).
\end{align}
We will prove in Section~\ref{sec:main_results} that $(U^n,\Delta^n)$ converges to $(u,\nabla u)$.

From now on we assume that $n \geq n_0(T,\|f\|_\lip)$ where $n_0(T,\|f\|_\lip) \in\nset^*$ is the integer given in Lemma~\ref{en:aeDB} in the appendix  and which 
automatically implies also existence and uniqueness of solutions because $n_0 > T\|f\|_\lip.$ 

% section notations (end)

\section{Scaled random walk and Wasserstein distance} % (fold)
\label{sec:scaled_random_walk_and_wasserstein_distance}
One starting point of our paper is the following result of Emmanuel Rio \cite{Rio09} (Theorem 2.1); see also \cite{Rio11}. This result covers, up to a generalization, the case where the generator vanishes, i.e. $f\equiv 0$.

Let $\psi$ be the convex function defined by $\psi(x)=e^{|x|}-1$. The Orlicz norm associated to this function $\psi$  of any real random variable $X$ is given by
\begin{equation*}
	\|X\|_\psi := \inf \{ a>0 : \e\left[\psi(X/a)\right] \leq 1\}, \qquad \inf\emptyset := +\infty.
\end{equation*}
Let us recall that, for any $r\geq 1$,
\begin{equation}\label{eq:orlicz}
	\sup_{x > 0} \left\{ \tfrac{x^r}{ \psi(x)} \right\} < +\infty, \qquad \|X\|_{\lp^r} \leq \left(\sup_{x > 0} \left\{ \tfrac{x^r}{ \psi(x)} \right\}\right)^{1/r} \, \|X\|_\psi.
\end{equation}
Let $X$ and $Y$ be two random variables end let us denote by $\mu$ the law of $X$ and by $\nu$ the law of $Y$. With the usual abuse of notation, the Wasserstein distance associated to $\psi$ is defined by
\begin{equation*}
	W_\psi(\mu,\nu)=W_\psi(X,Y) := \inf\left\{ \| X-Y\|_\psi : \text{law}(X)= \mu,\,  \text{law}(Y) = \nu\right\}.
\end{equation*}

Let $(X_k)_{k\geq 1}$ be an i.i.d. sequence of random variables with $\e\left[X\right]=0$, $\e\left[X^2\right]=1,$ and such that, for some $\sigma >0$, $\e\left[e^{\sigma |X|}\right]<+\infty$. Let $G$ be a  
standard normal random variable. In \cite[Theorem 2.1]{Rio09}, Emmanuel Rio proved that there exists a constant $C>0$ such that, for $n\geq 1$,
\begin{equation*}
	W_{\psi} \left( n^{-1/2}S_n,G\right) \leq C \, n^{-1/2}, \quad\text{where}\quad S_n=X_1+\ldots+X_n.
\end{equation*}
As a byproduct, for any  $r \ge  1$, there exists a constant $c_r>0$ such
that

\begin{equation*}
	W_{r} \left( n^{-1/2}S_n,G\right) \leq c_r \, n^{-1/2},
\end{equation*}
where $W_{r}$ stands for the $\lp^r$-Wasserstein distance
\begin{equation} \label{en: L-r-Wasserstein}
	W_r(\mu,\nu)=W_r(X,Y) := \inf\left\{ \e\left[|X-Y|^r\right]^{1/r} :  \text{law}(X)= \mu,\,  \text{law}(Y) = \nu \right\}. 
  \end{equation}
  We have also  the result of Kantorovich-Rubinstein, i.e.
\begin{equation}\label{eq:KR}
	W_1(\mu,\nu)=W_1(X,Y) = \sup\{ \e\left[f(X)\right]-\e\left[f(Y)\right] : \|f\|_\lip \leq 1\}.
\end{equation}

\begin{rem}
	We  could also consider the case where $0<r<1$ by using the fact that, in this case, $\e(|X-Y|^r)$ is a distance  (see the arguments in \cite[Section 7.1]{Ambrosio05}). 
 In general, we have $W_p(\mu,\nu)=W_q(\mu,\nu) $  for $0<p<q< \infty.$
\end{rem}

Let us start with a straightforward generalization of Rio's result.

\begin{prop}\label{en:fromRio}
	There exists a $C> 0$ such that, for all $x\in\rset$ and all $0\leq t\leq s\leq T$,
	\begin{equation*}
		W_{\psi}\left(B^{n,t,x}_s, B^{t,x}_s\right) \leq C\,\left(\frac{T}{n}\right)^{1/2}.
	\end{equation*}
\end{prop}

As a byproduct, taking into account~\eqref{eq:orlicz}, for any  $r \ge 1$, there exists a $c_r>0$ such that, for all $x\in\rset$ and  all $0\leq t\leq s\leq T$, 
\begin{equation}\label{eq:Wr}
	W_{r}\left(B^{n,t,x}_s, B^{t,x}_s\right) \leq c_r \left(\frac{T}{n}\right)^{1/2}.
\end{equation}

\begin{proof}[Proof of Proposition \ref{en:fromRio}]
	We have, for any $x \in \rset$ and all $0\leq t\leq s\leq T$,
	\begin{equation*}
		W_\psi\left(x + B^n_s-B^n_t, x+B_s-B_t\right) = W_\psi\left(B^n_s-B^n_t, B_s-B_t\right).
	\end{equation*}
	If $\us=\ut,$ then $B^n_t-B^n_s = 0,$ and we have
	\begin{equation*}
		W_\psi\left(B^n_s-B^n_t, B_s-B_t\right) = \| B_s-B_t\|_\psi = \sqrt{s-t}\, \|G\|_{\psi} \leq \sqrt{h}\, \|G\|_{\psi}.
	\end{equation*}
	
	Let us assume that $\ut<\us$ and let us write
	\begin{multline}\label{eq:yep}
		W_\psi\left(B^n_s-B^n_t, B_s-B_t\right) = W_\psi\left(B^n_{\us}-B^n_{\ut}, B_s-B_t\right) \\
		\leq W_\psi\left(B^n_{\us}-B^n_{\ut}, B_{\us}-B_{\ut}\right) + W_\psi \left(B_{\us}-B_{\ut},B_s-B_t\right).
	\end{multline}
	Let us treat each term separately. For the first one, Rio's result gives
	\begin{equation*}
		W_\psi\left( \frac{1}{\sqrt{n_s-n_t}}\, \sum_{k=n_t+1}^{n_s} \xi_k , G\right) \leq C\, \left(n_s-n_t\right)^{-1/2},
	\end{equation*}
	and multiplying by $\sqrt{n_s-n_t} \sqrt{h}$, we get, since $\sqrt{ h(n_s-n_t) } G$ is equal to $ B_{\us}-B_{\ut}$ in distribution,
	\begin{equation*}
		W_\psi\left(B^n_s-B^n_t, B_{\us}-B_{\ut}\right) \leq C\, \sqrt{h}.
	\end{equation*}
	Let us deal with the second term of \eqref{eq:yep}. Let $\beta(s,t) := \min(s-t,\us-\ut)$. Then 
	\begin{align*}
		W_\psi \left(B_{\us}-B_{\ut},B_s-B_t\right) & = W_\psi\left(\m N(0,\us-\ut),\m N(0,s-t)\right) \\
		& = W_\psi \left( \m N(0,\beta(s,t)),\m N(0,\beta(s,t)) * \m N(0, |s-t-(\us-\ut)|)\right) \\
		& \leq W_\psi \left(0, \m N(0, |s-t-(\us-\ut)|) \right) \\
		& = \sqrt{|s-t-(\us-\ut)|}\, \|G\|_{\psi}.
	\end{align*}
	But $|s-t-(\us-\ut)| \leq h,$ and this concludes the proof.
\end{proof}

Let us finish with a simple consequence of this result that we will use in the sequel. 

\begin{cor}\label{en:Hrate}
	Let $0<\ep\leq 1$ and let $g:\rset\fl\rset$ be an $\ep$-H\"older continuous function. Then there exists a $C>0$ depending on T such that, for all $x\in\rset$ and all $0\leq t\leq s\leq T$,
	\begin{equation*} %\label{eq:Hrate}
		\left| \e\left[g\left(B^{n,t,x}_s\right)\right]-\e\left[g\left(B^{t,x}_s\right)\right]\right| \leq C\, \|g\|_{\ep}\, n^{-\ep/2},
	\end{equation*}
	and, setting $\delta(t,s):=\max\left(s-t,\us-\ut\right)$,
	\begin{equation*}%\label{eq:Hrategrad}
		\left| \e\left[g\left(B^{n,t,x}_s\right)\left(B^{n,t,x}_s-x\right)\right]-\e\left[g\left(B^{t,x}_s\right)\left(B^{t,x}_s-x\right)\right]\right| \leq C\, \|g\|_{\ep}\,\delta(t,s)^{1/2}\, n^{-\ep/2}.
	\end{equation*}
\end{cor}

\begin{proof}
	Let $x\in\rset$ and $0\leq t\leq s\leq T$. 
	For any coupling $(X,Y)$ of $B^{n,t,x}_s$ and $B^{t,x}_s$, using Hölder's inequality when $0<\ep<1$,
	\begin{equation*}
		W_1\left(g\left(B^{n,t,x}_s\right),g\left(B^{t,x}_s\right)\right) \leq \e\left[\left| g\left(X\right)-g\left(Y\right)\right|\right] 
		\leq \|g\|_\ep\, \e\left[\left| X-Y\right|^{\ep}\right] \leq \|g\|_\ep\, \e\left[\left| X-Y\right|\right]^{\ep}.
	\end{equation*}
	Thus, we have, by~\eqref{eq:Wr} for $r=1$,
	\begin{equation*}
		W_1\left(g\left(B^{n,t,x}_s\right),g\left(B^{t,x}_s\right)\right) \leq \|g\|_\ep\, W_1\left(B^{n,t,x}_s,B^{t,x}_s\right)^{\ep} \leq  \|g\|_\ep\, c_1^{\ep}\, \left(\tfrac{T}{n}\right)^{\ep/2}.
	\end{equation*}
	 Choosing $f(x)=x$ in~\eqref{eq:KR}, this implies the first result.
	
	Let us prove the second assertion. We start by observing that, since $B_s-B_t$ and $B^n_s-B^n_t$ are centered random variables, we have, setting $h(y):=(g(x+y)-g(x))y$,
	\begin{align*}
		\e\left[g\left(B^{t,x}_s\right)\left(B^{t,x}_s-x\right)\right] & = \e\left[\left(g\left(x+B_s-B_t\right)-g(x)\right)(B_s-B_t)\right] = \e\left[h(B_s-B_t)\right], \\
		\e\left[g\left(B^{n,t,x}_s\right)\left(B^{n,t,x}_s-x\right)\right] & = \e\left[\left(g\left(x+B^n_s-B^n_t\right)-g(x)\right)\left(B^n_s-B^n_t\right)\right] = \e\left[h(B^n_s-B^{n}_t)\right].
	\end{align*}
	Let us remark that, for any real numbers $y$ and $z$, $|h(y)|\leq \|g\|_\ep\, |y|^{1+\ep}$, and using the fact that $|y-z|^{1-\ep} \leq |y|^{1-\ep}+|z|^{1-\ep}$,
	\begin{align*}
			|h(y)-h(z)| & \leq |g(x+y)-g(x)|\, |y-z| + |(g(x+y)-g(x))-(g(x+z)-g(x))| \, |z| \\
			& \leq \|g\|_\ep \left(|y-z|\, |y|^\ep + |y-z|^{\ep}\, |z| \right) \\
            &\leq \|g\|_\ep \left(|y-z|^{\ep}|y-z|^{1-\ep}\, |y|^\ep +
              |y-z|^{\ep}\, |z| \right)\\
            &\le \|g\|_\ep |y-z|^{\ep}\left(|y|+|y|^\ep\, |z|^{1-\ep}+ |z|\right).
	\end{align*}
	Young's inequality, $|y|^\ep\, |z|^{1-\ep} \leq \ep\,|y|+(1-\ep)\,|z| \leq |y|+|z|$, leads to
	\begin{equation}\label{eq:regh}
		|h(y)-h(z)| \leq 2\, \|g\|_\ep\,|y-z|^{\ep}\left(|y|+|z|\right).
	\end{equation}
	In the case where $\us=\ut$ we have
	\begin{align*}
		W_1\left(h(B^n_s-B^n_t),h(B_s-B_t)\right) & \leq \e\left[|h(B_s-B_t)|\right] \leq \|g\|_\ep\, \e\left[|B_t-B_s|^{(1+\ep)}\right] \\
		& \leq \|g\|_\ep \, (s-t)^{(1+\ep)/2} \, \e\left[|G|^{(1+\ep)}\right],
	\end{align*}
	where $ \text{law}(G) =\m N(0,1)$. Since $\us=\ut$ implies $(s-t)^{1/2} = \delta(t,s)^{1/2}$ and $(s-t)^{\ep/2} \leq h^{\ep/2},$  we have
	\begin{equation*}
		W_1\left(h(B^n_s-B^n_t),h(B_s-B_t)\right) \leq \|g\|_\ep\,\delta(t,s)^{1/2}\, (\tfrac{T}{n})^{\ep/2}
	\end{equation*}
	using the fact that $\e\left[|G|^{(1+\ep)}\right]\leq 1$.
	
	Let us turn to the case $\ut<\us$. For any coupling $(X,Y)$ of $B^n_s-B^n_t$ and $B_s-B_t$, using~\eqref{en: L-r-Wasserstein} and \eqref{eq:regh},
	\begin{align*}
		W_1\left(h(B^n_s-B^n_t),h(B_s-B_t)\right) & \leq \e\left[|h(X)-h(Y)|\right] \\
		 & \leq 2\,\|g\|_\ep\, \e\left[|X-Y|^{\ep}\, (|X|+|Y|)\right],
	\end{align*}
	and, by Hölder's inequality with $p=2/\ep$ and $q=2/(2-\ep)$,
	\begin{align*}
		W_1\left(h(B^n_s-B^n_t),h(B_s-B_t)\right) &\leq 2\,\|g\|_\ep\,  \e\left[|X-Y|^2\right]^{\ep/2}\, 
		 \,  \big ( \e [|X|^{2/(2-\ep)}]^{1-\ep/2}   + \e  [|Y|^{2/(2-\ep)}]^{1-\ep/2} \big) 
		 \\
		&\leq 2\, \|g\|_\ep \, \e\left[|X-Y|^2\right]^{\ep/2}  \left( (\us-\ut)^{1/2} +  (s-t)^{1/2}\right).
	\end{align*}
    From \eqref{en: L-r-Wasserstein} it follows that
	\begin{align*}
		W_1\left(h(B^n_s-B^n_t),h(B_s-B_t)\right) &\leq 2\, \|g\|_\ep\, W_2\left(B^n_s-B^n_t,B_s-B_t\right)^{\ep}\, \left((\us-\ut)^{1/2} + (s-t)^{1/2}\right) \\
      &\leq   4\, \, c_2^\ep\, \|g\|_\ep\, \left(\frac{T}{n}\right)^{\ep/2}\,\delta(t,s)^{1/2}, 
	\end{align*}
	where we have used~\eqref{eq:Wr} for $r=2$. 
	
	Thus, for $0\leq t\leq s\leq T$, 
	\begin{equation*}
		W_1\left(h(B^n_s-B^n_t),h(B_s-B_t)\right) \leq C\, \|g\|_\ep\, n^{-\ep/2}\,\delta(t,s)^{1/2},
	\end{equation*}
	and the result follows as before by choosing $f(x)=x$ in~\eqref{eq:KR}.
\end{proof}

% section scaled_random_walk_and_wasserstein_distance (end)

\section{Regularity results on  \texorpdfstring{$u$, $U^n,$ $\nabla u$}{u} and  \texorpdfstring{$\Delta^n$}{D}} % (fold)
\label{sec:regularity_results_on_u_u_n_nabla_u_and_delta_n}
Let us start by  known regularity properties of the function $u$ that follow from classical a priori estimates for BSDEs.

\begin{lemme}
	\label{en:regu}
	Under Assumption \ref{H:A1} there exists a constant  $C>0$ depending on $(T,\ep,f,g)$ such that, for all $(t,x)\in[0,T]\times\rset$,
	\begin{equation*}
		|u(t,x)| \leq C\, (1+|x|)^\ep, \quad \qquad \| u(t,\cdot) \|_\ep \leq
        C, \quad \qquad \| u(\cdot,x) \|_{\ep/2} \leq C\, (1+|x|)^\ep.
	\end{equation*}
  \end{lemme}

  \begin{proof}
    The first two results follow directly from classical a priori estimates for BSDEs, see e.g. \cite[Proposition 2.1]{elkaroui}. The last one ensues from the following upper bound: for any real $x$ and for $0\leq r\leq t\leq T$,
    \begin{align*}
      |u(r,x)-u(t,x)| &  = |\e[Y^{r,x}_r] -\e[Y^{t,x}_t] | \\
	  & \le |\e[ Y^{r,x}_r-Y^{r,x}_t]|+ |\e[Y^{r,x}_t-Y^{t,x}_t]|\\
      & \le \int_r^t\e[ |f(s,B^{r,x}_s,Y^{r,x}_s,Z^{r,x}_s)| ]\, ds +
        \e[|Y^{r,x}_t-Y^{t,x}_t|]\\
      & \le C \int_r^t \e[ 1+|B^{r,x}_s|^{\varepsilon}+|Y^{r,x}_s|+|Z^{r,x}_s|]\, ds +
        \e[|Y^{r,x}_t-Y^{t,x}_t|].
    \end{align*}
    Since the norm in $\mathcal{S}^2 \times \mathcal{H}^2$ of $(Y^{r,x},Z^{r,x})$ is of order $(1+|x|)^\ep$,
    we use Cauchy-Schwarz inequality to bound the first term and a priori
    estimates enable  (similarly as in the proof of \cite[Proposition 4.1]{elkaroui}) to bound the second term.
  \end{proof}

 Next we extend \cite[Theorem 3.2]{Zha05} to the case where $f(t,\cdot,y,z)$ is Hölder continuous.

\begin{lemme}
	\label{en:regnabus}
	Recall  the notation \eqref{eq:defF} and let Assumption \ref{H:A1} hold. 
	
	\begin{enumerate}[{\rm (a)}]
		\item \label{ma_zhang_formula} 
		The function $u$ belongs to $\m C^{0,1}([0,T[\times\rset)$ and, for all $(t,x) \in [0,T[
  \times \rset$, we have,
  \begin{align}\label{eqZ}
    Z^{t,x}_s = \nabla u(s, B^{t,x}_s)  \quad \text{for a.e.} \,\,
    (s,\omega)\in [t,T[\times \Omega, 
  \end{align}
  as well as \eqref{eq:nabla-u} i.e.
  \begin{equation*}
    \nabla u(t,x)=\e\left(g(B^{t,x}_T)\frac{B_T-B_t}{T-t}\right)+\e\left(\int_t^T
      F(s,B^{t,x}_s)\frac{B_s-B_t}{s-t}ds\right).
  \end{equation*}
  \item Moreover, there exists a constant  $C>0$ depending on $(T,\ep,f,g)$ such that, 
  \begin{enumerate}[{\rm (i)}] 
  \item \label{nabla u-diff} $ \|\nabla u(t, \cdot) \|_\ep \le\frac{C}{\sqrt{T-t}}$  for all $t \in [0,T[,$
  \item \label{nabla u-bound} $|\nabla u(t,x)|\le \frac{C}{(T-t)^{(1-\ep)/2}}$ for all $(t,x) \in [0,T[ \times \rset.$ 
  \end{enumerate}
\end{enumerate}
\end{lemme}

Consequently, for $\e_r := \e [\, \cdot\, |\m F_r]$,
\begin{align}\label{eq: nabla-u-b} 
 \nabla u(r, B^{t,x}_r)&=\e_r\left(g(B^{t,x}_T)\frac{B_T-B_r}{T-r}\right)+\e_r\left(\int_r^T
                     F(s,B^{t,x}_s)\frac{B_s-B_r}{s-r}ds\right) \,\,  \text{ a.s. for } \, r \in  [t,T[.
\end{align}

\begin{proof}[ Proof of Lemma \ref{en:regnabus}]
	The proof is divided into two steps.
	
\noindent\emph{Step 1. We assume in addition that $f$ is Lipschitz continuous w.r.t. $x$.}\quad Then according to \cite{Zha05}, we have only the second point to prove and we know that, for some constant $C$,
\begin{equation}\label{eq:nabla-u-estimate}  
|\nabla u(t,x)|\le \frac{C(1+|x|)}{\sqrt{ T-t}} \quad \text{   on }   [0,T[ \times \rset.
\end{equation}

\eqref{nabla u-diff} The representation \eqref{eq: nabla-u-b} yields to
 \begin{align} \label{L2-estimate}
	 \notag
 \| \nabla u(r, B^{t,x}_r) - \nabla u(r, B^{t,y}_r)\|_{L^2}  
 &\le \left \|    \e_r\left(\left (g(B^{t,x}_T) - g(B^{t,y}_T) \right )\frac{B_T-B_r}{T-r}\right) \right \|_{L^2}  \\
 & \quad +  \int_r^T  \left \|  \e_r\left(
                     \left(F(s,B^{t,x}_s) -F(s,B^{t,y}_s)\right) \frac{B_s-B_r}{s-r} \right)   \right \|_{L^2}   ds. 
  \end{align}
Since $g$ is $\ep$-Hölder  continuous we get
$$ \left | \e_r\left((g(B^{t,x}_T) - g(B^{t,y}_T))\frac{B_T-B_r}{T-r}\right)\right | \le \|g\|_{\ep}  |x-y|^\ep \frac{1}{\sqrt{T-r}}.$$
Similarly,  we obtain by the conditional Cauchy-Schwarz inequality the estimate
$$ \left \|  \e_r\left(
                     (F(s,B^{t,x}_s) -F(s,B^{t,y}_s)) \frac{B_s-B_r}{s-r} \right)   \right \|_{L^2} \le \left \| 
                     F(s,B^{t,x}_s) -F(s,B^{t,y}_s)   \right \|_{L^2}  \frac{1}{\sqrt{s-r}}.   $$
Using \eqref{eq:prop_f} for $f$, we have
\begin{multline*}
	| F(s,B^{t,x}_s) -F(s,B^{t,y}_s)| \\
	\leq \|f_x\|_{\ep}|x-y|^{\ep}+\|f_y\|_{\lip}| u(s,B^{t,x}_s) -u(s,B^{t,y}_s) |+\|f_z\|_{\lip}|  \nabla u(s,B^{t,x}_s)-  \nabla u(s,B^{t,y}_s)|, 
\end{multline*}
  and the Hölder continuity of $u$ stated in Lemma~\ref{en:regu} yields
 \begin{multline} \label{F-difference} 
	 | F(s,B^{t,x}_s) -F(s,B^{t,y}_s)| \\
 \le( \|f_x\|_{\ep}+  C\|f_y\|_{\lip}) |x -y |^\ep+\|f_z\|_{\lip}|  \nabla u(s,B^{t,x}_s)-  \nabla u(s,B^{t,y}_s)|.
 \end{multline}
 By combining the above estimates we conclude from \eqref{L2-estimate}  that 
\begin{align*} 
 \| \nabla u(r, B^{t,x}_r) - \nabla u(r, B^{t,y}_r)\|_{L^2}  & \le \|g\|_{\ep}  \frac{   |x-y|^\ep }{\sqrt{T-r}} + 2\, (\|f_x\|_{\ep}+ C\, \|f_y\|_{\lip})\, |x -y |^\ep \sqrt{T-r} \\
 & \quad +   \int_r^T  \left \|   \nabla u(s,B^{t,x}_s)-  \nabla u(s,B^{t,y}_s)  \right \|_{L^2} \frac{\|f_z\|_{\lip}}{\sqrt{s-r}}  ds\\
 &\le C_1  \frac{|x-y|^\ep}{\sqrt{T-r}} +  \int_r^T  \left \|   \nabla u(s,B^{t,x}_s)-  \nabla u(s,B^{t,y}_s)  \right \|_{L^2} \frac{ \|f_z\|_{\lip}}{\sqrt{s-r}}  ds 
  \end{align*}
for $C_1:= \|g\|_{\ep}  +2T (\|f_x\|_{\ep}+C\, \|f_y\|_{\lip}) .$   Because of \eqref{eq:nabla-u-estimate} we have 
\begin{equation*}
	\left \|   \nabla u(s,B^{t,x}_s)-  \nabla u(s,B^{t,y}_s)  \right \|_{L^2} \le  C_2 \frac{1+|x|+|y|}{  (T-s)^{1/2}}
\end{equation*}   
with $C_2=C_2(C,T)>0$.  Hence we may apply 
 Gronwall's lemma  (Lemma \ref{volterra_gronwall}) and get 
 \begin{align*} \| \nabla u(r, B^{t,x}_r) - \nabla u(r, B^{t,y}_r)\|_{L^2}  \le  C_1 c_0 \frac{|x-y|^\ep}{\sqrt{T-r}},  \end{align*}
for some $c_0 =c_0(T, \|f_z\|_{\lip})>0.$ Especially, for $r=t$ this implies
 \begin{align*} 
  | \nabla u(t, x) - \nabla u(t, y)|   &\le C \frac{ |x-y|^\ep}{\sqrt{T-t}}
  \end{align*}
for some  $C=C(T,\ep,f, g)>0.$  \smallskip

  \eqref{nabla u-bound} We first notice that for any $\ep$-H\"older continuous function $k$ and for all $0 \le t<s\le T$ we have 
\begin{equation} \label{holder-estimate}
\left|\e \left[ k(B^{t,x}_s)\frac{B_s-B_t}{s-t}\right]\right| = \left|\e \left[(k(B^{t,x}_s)-k(x))\frac{B_s-B_t}{s-t}\right]\right | \le  \frac{\|k\|_{\ep}}{(s-t)^{(1-\ep)/2}}.
\end{equation}
Therefore, we obtain from \eqref{eq:nabla-u}  that
 \begin{align*}
  |\nabla u(t,x)|&\le \left |\e\left(g(B^{t,x}_T)\frac{B_T-B_t}{T-t}\right) \right | +
 \left | \e\left(\int_t^T  F(s,B^{t,x}_s)\frac{B_s-B_t}{s-t}ds\right) \right |  \\
&\le  \frac{\|g\|_{\ep}}{(T-t)^{(1-\ep)/2}}  +\int_t^T  \frac{\|F(s, \cdot)\|_{\ep}}{(s-t)^{(1-\ep)/2}}ds.
\end{align*}
Using \eqref{F-difference} for $s=t$ and taking into account that $\nabla u$ satisfies   \eqref{nabla u-diff} we get
\begin{align} \label{eq:regF}
 \|F(s,\cdot)\|_\ep \le  \frac{C}{\sqrt{T-s}}
\end{align}
for some  $C=C(T,\ep,f, g)>0.$ 
This finishes the proof of the first step.

\emph{Step 2. General case.}\quad The proof relies on a regularization procedure and is postponed to appendix~\ref{A: generator Hoelder}.
\end{proof}

\begin{rem}
\label{en:Zreg}
	From now on we will always use the  continuous version of $Z^{t,x}_s$ given by $\nabla u(s,B^{t,x}_s)$.
\end{rem}

\begin{lemme}\label{lem1} For all $(t,x) \in [0,T[ \times \rset$ and for
    $n\ge n_0\in \nset^*$,  with $n_0$ defined as in Lemma \ref{en:aeDB}, we have
 \begin{enumerate}[{\rm(i)}]
 \item \label{U-n-bound}
   $ |U^n(t,x)|\le C(1+|x|)^{\ep},$  
   \item $ |\Delta^n (t,x)|\le \frac{C_n}{(T-t)^{(1-\ep)/2}},$
 \end{enumerate}
  where $C>0$ depends on $(T, \ep,f,g)$ and $C_n>0$ depends on $(T, \ep,f,g,n).$
\end{lemme}

\begin{proof}
  The result on $U^n$ ensues from Lemma \ref{en:aeDB}, by choosing
     $\overline{f}=0$ and $\overline{g}=0$.
 Let us prove the result on $\Delta^n$.   
 By \eqref{Delta-n} and \eqref{nabla} we have that 
\begin{align*}
\Delta^n(t,x) & = [\nabla^n U^n](\ut+h,x) \\
& = \tfrac{1}{2\sqrt{h}}  ( U^n(\ut+h,x + \sqrt{h}) - U^n(\ut+h,x -\sqrt{h})) \\
& = \e \left [ U^n(\ut+h, B^{n,\ut,x}_{\ut+h}  )   \tfrac{B^{n,\ut,x}_{\ut+h} -x}{h}  \right ].
\end{align*}
We want to use \eqref{U-n}, where we realize that  
$$   \e \left [   g\left(B^{n,\ut+h , B^{n,\ut,x}_{\ut+h}}_T\right)  \tfrac{B^{n,\ut,x}_{\ut+h} -x}{h}  \right ] =   \e \left [   g\left(B^{n,t,x}_T\right)  \tfrac{B^{n,\ut,x}_{\ut+h} -x}{h}  \right ]  = \e\left[g(B^{n,t,x}_T)\frac{B^{n,t,x}_T-x}{T-\ut} \right]. $$
A similar argument can be used for the integral expression so that we get 
\begin{equation}\label{eq:gradnu}
	\Delta^n(t,x) = \e\left[g(B^{n,t,x}_T)\frac{B^{n,t,x}_T-x}{T-\ut} \right] +\e\left[\int_{\ut+h}^T f\left(\os,\Theta^{n,t,x}_s\right)\frac{B^{n,t,x}_s-x}{\us-\ut} \,ds\right].
\end{equation}
Then
\begin{align*} %\label{eq6}
  \Delta^n(t,x)=&\, \e\left((g(B^{n,t,x}_T)-g(x))\frac{B^n_T-B^n_t}{T-\ut}\right)\notag\\
                 &+\e\left(\int_{\ut+h}^T
  (f(\os,\Theta^{n,t,x}_s)-f(\os,x,U^n(s,x),\Delta^n(s,x)))\frac{B^n_s-B^n_t}{\us-\ut}ds\right)\\
  =:\,& G+F. \notag
\end{align*}

Since $g$ is $\ep$-Hölder, $|G|$ is bounded by
$\frac{\|g\|_{\ep}}{(T-t)^{(1-\ep)/2}}$. Concerning the second term, we get, since $f$
satisfies \eqref{eq:prop_f},
\begin{align*}
  &f(\os,\Theta^{n,t,x}_s)-f(\os,x,U^n(s,x),\Delta^n(s,x))\\
  &\le \|f_x\|_{\ep}|B^n_s-B^n_t|^{\ep}+\|f_y\|_{\lip}|U^n(s,B^{n,t,x}_s)-U^n(x,s)|+\|f_z\|_{\lip}|\Delta^n(s,B^{n,t,x}_s)-\Delta^n(s,x)|.
\end{align*}
We will use  that $U^n(t,x)$ and $\Delta^n(t,x)$ are
$\ep$-Hölder continuous in $x$, i.e.
\begin{align*} %\label{eq7}
  |U^n(t,x)-U^n(t,y)|+|\Delta^n(t,x)-\Delta^n(t,y)|\le c(h)|x-y|^{\ep},
\end{align*}
where $c(h)$ tends to infinity when $h$ tends to $0$. For $U^n$, Lemma~\ref{en:aeDB} with $(\bar x, \bar g,\bar f)=(y,g,f)$ gives
\begin{align}\label{eq7a}
  |U^n(t,x)-U^n(t,y)|\le c_0|x-y|^{\ep}
\end{align}
while for $\Delta^n$ this is an immediate consequence of 
Remark \ref{rem3} and  \eqref{eq7a} with  $c(h)={c_0+ \frac{c_0}{\sqrt{h}}}.$   Then
\begin{align*}
  |F| &\le
    \left (  \|f_x\|_{\ep}+c(h)\|f_y\|_{\lip}+c(h)\|f_z\|_{\lip}  \right ) \int_{\ut+h}^T
 \e\left(\frac{|B^n_s-B^n_t|^{1+\ep}}{\us-\ut}\right)ds\\
  &\le \left ( \|f_x\|_{\ep}+c(h)\|f_y\|_{\lip}+c(h)\|f_z\|_{\lip} \right )  \int_{\ut+h}^T \frac{1}{(\us-\ut)^{(1-\ep)/2}}ds.
\end{align*}
Since $\frac{1}{(\us-\ut)^{(1-\ep)/2}}\le \frac{1}{(s-(\ut+h))^{(1-\ep)/2}}$ for $s \in \, ]\ut+h,T]$
we get that $|F|\le C_n$.

\end{proof}

\begin{prop}\label{en:timenabu}
  Under~\ref{H:A1}, there exists a constant  $C>0$ depending on $(T,\ep,f,g)$ such that, for all $x \in \rset$,
  \begin{equation}
  	\label{2}
	|\nabla u(t,x) -\nabla u(r,x) | \le C \frac{(t-r)^{\ep/2}}{\sqrt{T-t}} \quad \text{ for all}\,\, \,\, 0\le r <t <T.
  \end{equation}  
\end{prop}

\begin{proof} From Lemma~\ref{en:regnabus}, we know that, for $(t,x)\in[0,T[\times\rset$,
	\begin{equation}
		\label{eq:withH}
		\nabla u(t,x) = \e\left[ g\left(B_T^{t,x}\right) H(t,T)  + \int_t^T F\left(s,B_s^{t,x}\right)H(t,s)\, ds\right],
	\end{equation}
	where we have set 
	$$
	H(t,s):= \frac{B_s-B_t}{s-t},\quad \text{ for } \, 0\le t<s\le T.
	$$
It holds
$\e [H(t,s)]  =0 $  and  $\|H(t,s) \|_{\lp^2} = \frac{1}{\sqrt{s-t}}$.  We also have, for $0\leq r\leq t<s$,
 \begin{align*}
\e|H(r,s)-H(t,s)|^2 &=\frac{1}{s-r} - \frac{2}{(s-t)(s-r)}
                      \e[(B_s-B_r)(B_s-B_t)]+ \frac{1}{s-t} \\
&=\frac{1}{s-t} - \frac{1}{s-r} =\frac{t-r}{(s-r)(s-t)}.
 \end{align*}
Let us observe that, for $0\le r<t<s \le T$ and any $\ep$-H\"older continuous function $h$, it holds
\begin{equation*}
	\left| \e \left[h(B_s^{t,x}) H(t,s)\right] - \e\left[h(B_s^{r,x}) H(r,s)\right] \right| \le  2\,    \|  h \|_\ep   \frac{(t-r)^{\ep/2}}{(s-t)^{1/2}}.
\end{equation*}
Indeed, we have  
 \begin{multline*}
\left| \e \left[h(B_s^{t,x}) H(t,s)\right] - \e\left[h(B_s^{r,x}) H(r,s)\right] \right| \\ 
 \le  \e \left| \left[h(B_s^{t,x})  - h(B_s^{r,x} )\right] H (t,s) \right| +   \e \left| \left[h(B_s^{r,x})  - h(x)\right] [H (t,s) - H (r,s)] \right| \\
 \le \| h\|_\ep \left( \e \left[|B_t -B_r |^\ep |H( t,s)|\right] + \e \left[|B_s-B_r|^\ep |H(t,s)-H(r,s) |\right]  \right),
 \end{multline*}
 and, from Cauchy-Schwarz inequality, we deduce that
 \begin{align*}
	 \left| \e \left[h(B_s^{t,x}) H(t,s)\right] - \e\left[h(B_s^{r,x}) H(r,s)\right] \right| 
& \le \| h\|_\ep \left [    \frac{(t-r)^{\ep/2}}{(s-t)^{1/2}} + (s-r)^{\ep/2}  \frac{(t-r)^{1/2}}{(s-r)^{1/2} (s-t)^{1/2}}   \right ] \\
& =  \| h\|_\ep   \frac{(t-r)^{\ep/2}}{(s-t)^{1/2}}  \left [ 1  +    (t-r)^{\frac{1}{2}-\frac{\ep}{2}} (s-r)^{\frac{\ep}{2}-\frac{1}{2}}   \right ] \\
& \leq 2\,\| h\|_\ep\,   \frac{(t-r)^{\ep/2}}{(s-t)^{1/2}}. 
 \end{align*}

Coming back to~\eqref{eq:withH}, we write, for $0\leq r\leq t<T$,
\begin{align*}
	\nabla u(r,x)  & = \e\left[ g\left(B_T^{r,x}\right) H(r,T)  + \int_t^T F\left(s,B_s^{r,x}\right)H(r,s)\, ds \right] \\
	 &  \quad + \e\left[\int_r^t \left( F\left(s,B_s^{r,x}\right)-F(s,x)\right) H(r,s)\, ds\right],
\end{align*}
to have, taking into account the fact that $\|F(s,\cdot)\|_\ep \leq C (T-s)^{-1/2}$ by 
 \eqref{eq:regF},  
\begin{align*}
     | \nabla u(t,x)-\nabla u(r,x)| &\le C \left[ (t-r)^{\ep/2} \left( \frac{ \|g \|_\ep}{(T-t)^{1/2}}+ \int_t^T \frac{\|F(s,\cdot)\|_\ep}{ (s-t)^{1/2}}ds    \right)  
          +   \int_r^t \frac{ \|F(s,\cdot) \|_\ep}{ (s-r)^{(1-\ep)/2}} ds \right], \\
		  \\
		  & \le   C\Bigg [(t-r)^{\ep/2} \left( \frac{ \|g \|_\ep}{(T-t)^{1/2}}+ \int_t^T \frac{ds}{ (T-s)^{1/2} (s-t)^{1/2}}  \right)  \notag \\
		  & \quad +     \int_r^t \frac{ds}{ (T-s)^{1/2}  (s-r)^{(1-\ep)/2} } \Bigg] \\
		  & \le   C\left [(t-r)^{\ep/2} \left(\frac{ \|g \|_\ep}{(T-t)^{1/2} }  +  B(\tfrac{1}{2},\tfrac{1}{2}) \right)  +   \frac{(t-r)^{(\ep+1)/2}}{ (T-t)^{1/2}}\right ]  .
\end{align*}
\end{proof}

\section{Main results} % (fold)
\label{sec:main_results}
In this section, we state the  main result of this paper which gives the rate of convergence in the Wasserstein distance between the solution to the BSDE~\eqref{eq:mainBSDE} and the solution to the BSDE driven by the scaled random walk~\eqref{eq:mainBSDEn}. 
For the following we want to remind the reader of Remark~\ref{en:Zreg}. 
 
\begin{thm}\label{en:mainBSDE}
	Under \ref{H:A1}, for any $r\in [1,\infty [$, there exists a constant $C_r>0$   depending at most on $(T,\alpha,\ep,f,g,r)$ such that for all $x\in\rset$, 
 \begin{enumerate}[{\rm(i)}]
 \item	$	W_r\left(Y^{n,t,x}_s,Y^{t,x}_s\right) \leq C_r\, (1+|x|)^\ep \, n^{-(\alpha\wedge \frac{\ep}{2})} \quad \text{ for all }\,\, 0\le t \le s\le T, $
 \item    $W_r\left(Z^{n,t,x}_s,Z^{t,x}_s\right) \leq C_r\, \frac{(1+|x|)^\ep}{\sqrt{T-s}} \, n^{-(\alpha\wedge \frac{\ep}{2})} \quad \text{ for all} \,\,s \in [t,T[. $
\end{enumerate}		
\end{thm}

This result is  a consequence of  the following proposition which gives the rate of the point-wise convergence of $U^n,$ solution to~\eqref{eq:mainPDEn}, towards the 
solution $u$ of the semilinear heat equation~\eqref{eq:mainPDE}. 

 \begin{prop}\label{en:mainPDE}
	Under \ref{H:A1} there exists a constant $C>0$  depending at most on $(T,\alpha,\ep,f,g)$  such that
	 \begin{enumerate}[{\rm(i)}]
  \item $ %\label{eq4}
 		|u(t,x)-U^n(t,x)|  \leq  C\, (1+|x|)^\ep \, n^{-(\alpha\wedge \frac{\ep}{2})} \,\, \text{ for all } \,\,(t,x)\in\rset\times [0,T], $
 \item $%\label{eq5}
 		|\nabla u(t,x)-\Delta^n(t,x)| \leq C \, \frac{(1+|x|)^\ep}{\sqrt{T-t}} \, n^{-(\alpha\wedge \frac{\ep}{2})} \,\, \text{ for all } \,\, (t,x)\in\rset\times [0,T[. $
 	\end{enumerate}	%
  \end{prop}

\begin{proof} 
We split the proof into three parts. We begin by studying $|u-U^n|$, we proceed by obtaining an estimate for $\nabla u - \Delta^n,$ and then we conclude with a Gronwall argument.

\subsection*{Estimate for $|u-U^n|$} % (fold)
%\label{sub:estimate_for_u_u_n}
%
From \eqref{eq: u} and   \eqref{U-n} we conclude that 
	\begin{equation}\label{eq:go}
		\begin{split}
			|u(t,x)-U^n(t,x)| & \leq \left| \e\left[g\left(B^{t,x}_T\right)\right] - \e\left[g\left(B^{n,t,x}_T\right) \right] \right|  \\
			 & + \left| \e\left[\int_t^T F\left(s,B^{t,x}_s\right) ds\right] - \e\left[\int_{\ut}^T f\left(\os,\Theta^{n,t,x}_s\right) ds\right] \right|.
		\end{split}	
	\end{equation}
	
Let $F_n$ be the function given by
\begin{equation}\label{eq:defFn}
	 F_n(s,x):=f(s,x,U^n(s,x),\Delta^n(s,x)) \quad \text{ for } \,\, (s,x) \in [0,T[\times \rset. 
\end{equation}
 Using the notation \eqref{theta-n} we also have that  $F_n(s,B^{n,t,x}_s)=f(s,\Theta^{n,t,x}_s)$. With this notation in hand, we have, 
 taking into account \eqref{eq:RY} and \eqref{eq:RZ},
	\begin{align*}
		\e\left[\int_{\ut}^T f\left(\os,\Theta^{n,t,x}_s\right) ds\right] & = \e\left[\int_{\ut}^T \left(f\left(\os,\Theta^{n,t,x}_s\right)-f\left(s,\Theta^{n,t,x}_s\right)\right)\right] +
		\e\left[\int_{t}^T F_n\left(s,B^{n,t,x}_s\right) ds\right] \\
		&\quad + \e\left[\int_{\ut}^t f\left(s,B^{n,t,x}_s,Y^{n,t,x}_s,Z^{n,t,x}_s\right) ds\right].
	\end{align*}
	In view of the regularity of $f$ in time, we have 
	\begin{equation*}
		\left| \e\left[\int_{\ut}^T \left(f\left(\os,\Theta^{n,t,x}_s\right)-f\left(s,\Theta^{n,t,x}_s\right)\right)ds\right] \right| \leq C\, n^{-\alpha}.
	\end{equation*}
	Moreover,  the Cauchy-Schwarz inequality leads to
	\begin{equation*}
		\left| \e\left[\int_{\ut}^t f\left(s,B^{n,t,x}_s,Y^{n,t,x}_s,Z^{n,t,x}_s\right) ds\right] \right| \leq (t-\ut)^{1/2}\, \e\left[\int_{\ut}^t \left|f\left(s,B^{n,t,x}_s,Y^{n,t,x}_s,Z^{n,t,x}_s\right)\right|^2\, ds\right]^{1/2},
	\end{equation*}
	and, taking into account the growth of $f$, we have
	\begin{align*} 
		&\left| \e\left[\int_{\ut}^t f\left(s,B^{n,t,x}_s,Y^{n,t,x}_s,Z^{n,t,x}_s\right) ds\right] \right| \\
		&\quad \quad \quad \leq C\,h^{1/2}\, \e\left[\int_{\ut}^T \left(1+|B^{n,t,x}_s|^{2\ep}+|Y^{n,t,x}_s|^2+|Z^{n,t,x}_s|^2\right) ds\right]^{1/2} \\
		&\quad \quad \quad \leq C \, n^{-1/2}\, (1+|x|)^{\ep},
	\end{align*}
    where we have used Lemma \ref{en:aeDB} to get
    \begin{align*}
      \e\left(\sup_{\ut \le s \le T} |Y^{n,t,x}_s|^2+\int_{\ut}^T |Z^{n,t,x}_s|^2
      ds\right)\le & C\e \left(|g(B^{n,t,x}_T)|^2+\int_{\ut}^T
      |f(s,B^{n,t,x}_{ s-},0,0)|^2ds\right)\\
      & \le C(1+|x|)^{2\ep}.
    \end{align*} 
	
	 Coming back to~\eqref{eq:go}, we derive the following inequality
	\begin{align} \label{u-minus-U-n}
			|u(t,x)-U^n(t,x)| & \leq \left| \e\left[g\left(B^{t,x}_T\right)\right] - \e\left[g\left(B^{n,t,x}_T\right) \right] \right| \notag \\
			 &\quad + \left| \int_t^T \e\left[ F\left(s,B^{t,x}_s\right) \right] ds - \int_{t}^T \e\left[F_n\left(s,B^{n,t,x}_s\right)\right] ds \right| \notag \\
			 &\quad + C\, (1+|x|)^\ep \, n^{-(\alpha \wedge \frac{1}{2})}.
	\end{align}
	From Corollary~\ref{en:Hrate} we get
	\begin{equation*}
		\left| \e\left[g\left(B^{t,x}_T\right)\right] - \e\left[g\left(B^{n,t,x}_T\right) \right] \right| \leq C\, \|g\|_\ep \, n^{-\frac{\ep}{2}}.
	\end{equation*}
We split the second term on the RHS of \eqref{u-minus-U-n} into two parts
	\begin{multline*}
		\left| \int_t^T \e\left[ F\left(s,B^{t,x}_s\right) \right] ds - \int_{t}^T \e\left[F_n\left(s,B^{n,t,x}_s\right)\right] ds \right|  \\ \leq   \int_t^T \left| \e\left[ F\left(s,B^{t,x}_s\right) \right] - \e\left[F\left(s,B^{n,t,x}_s\right)\right] \right| ds  + \int_t^T \e\left[ |F-F_n|\left(s,B^{n,t,x}_s\right) \right] ds. 
	\end{multline*}
	Since $F$ has the regularity~\eqref{eq:regF}, Corollary~\ref{en:Hrate} gives
\begin{equation*}
	\int_t^T \left| \e\left[ F\left(s,B^{t,x}_s\right) \right] - \e\left[F\left(s,B^{n,t,x}_s\right)\right] \right| ds  \leq C\, \int_t^T \frac{ds}{\sqrt{T-s}} \, ds \, n^{-\frac{\ep}{2}} \leq C\, n^{-\frac{\ep}{2}}.
\end{equation*}
By the above estimates  we derive from \eqref{u-minus-U-n} the  inequality
\begin{equation}\label{eq:presqueu}
	|u(t,x)-U^n(t,x)|  \leq C\, (1+|x|)^\ep \, n^{-(\alpha \wedge \frac{\ep}{2})} + \int_t^T \e\left[ |F-F_n|\left(s,B^{n,t,x}_s\right) \right] ds.
\end{equation}
Coming back to the definition of $F$ and $F_n$  (see \eqref{eq:defF} and \eqref{eq:defFn})
and using the Lipschitz continuity of $f$ with respect to $(y,z)$, we have
\begin{equation*}
	 |F-F_n|\left(s,B^{n,t,x}_s\right)  \leq \|f\|_\lip \left[|u-U^n|(s,B^{n,t,x}_s)+ |\nabla u - \Delta^n|(s,B^{n,t,x}_s)\right].
\end{equation*}
Setting for simplicity, for $s \in [0,T]$,
\begin{equation*} %\label{eq:defVW}
	    \beta_n(s) := \sup_{x\in\rset}\left\{ \frac{|u-U^n|(s,x)}{(1+|x|)^\ep}\right\} 
        \quad	   
	    \text{ and } \quad \gamma_n(s):=\sup_{x\in\rset}\left\{ \frac{|\nabla u-\Delta^n|(s,x)}{(1+|x|)^\ep}\right\}
\end{equation*}	
  for $s \in [0,T[$, Lemma \ref{en:regu}, Lemma \ref{en:regnabus} and Lemma~\ref{lem1} imply  that, for some $C>0$  and $C_n >0,$
   $$ \beta_n(s) \le C \quad \text{ for} \,\, s \in [0,T] \quad  \text{ and } \quad \gamma_n(s) \le \frac{C_n}{(T-s)^{(1-\ep)/2}} \quad \text{ for }\,\, s \in [0,T[,   $$
respectively. We deduce the following estimate
\begin{equation} \label{FminusF-n}
	 |F-F_n|\left(s,B^{n,t,x}_s\right)  \leq  \|f\|_\lip (1+|B^{n,t,x}_s|)^\ep \left(\beta_n(s) +  \gamma_n(s) \right)
\end{equation}
 and get, coming back to~\eqref{eq:presqueu}, for $0\leq t\leq T$ and for any $x\in\rset$,
\begin{equation*}
	|u(t,x)-U^n(t,x)|  
	\leq C\, (1+|x|)^\ep  \left(n^{-(\alpha \wedge \frac{\ep}{2})} + \int_t^T \left( \beta_n(s)+\gamma_n(s) \right) ds\right).
\end{equation*} 
We end up with the inequality 
\begin{equation*} %\label{V-n-estimate}
	\beta_n(t)  \leq C\left( n^{-(\alpha\wedge \frac{\ep}{2})} + \int_t^T \left( \beta_n(s)+\gamma_n(s) \right) ds \right), \quad  t \in [0,T],
\end{equation*}
and since $\gamma_n$ belongs to $\lp^1[0,T],$ Gronwall's inequality  (Lemma \ref{lem:gronwall})  gives
\begin{equation}\label{eq:almostY}
\beta_n(t) \leq C\left( n^{-(\alpha\wedge \frac{\ep}{2})} + \int_t^T \gamma_n(s)  ds\right), \quad t \in [0,T].
\end{equation} 
\subsection*{Estimate for $|\nabla u -\Delta^n|$} % (fold)
%\label{sub:estimate_for_nabla_u_delta_n}
%
In order to take advantage of the previous inequality, we need to estimate $\gamma_n(s)$. To do this, we  use the  representations \eqref{eq:nabla-u} and \eqref{eq:gradnu}.
We will divide the study into two parts
\begin{align*} %\label{eq2}
  | \nabla u(t,x)- \Delta^n (t,x)|\le |\mbox{ $g$ difference }| + |\mbox{ $f$  difference }|. 
\end{align*}

\paragraph{Study of the $g$ difference} % (fold)
%\label{par:study_of_the_g_difference}
%
We have
\begin{align} \label{g-difference}
  &\e\left[g\left(B^{t,x}_T\right)\frac{B^{t,x}_T -
  x}{T-t}\right]-\e\left[g\left(B^{n,t,x}_T\right)\frac{B^{n,t,x}_T-x}{T-\ut}\right] \notag\\
  & \qquad = \e\left[g(x+B_T-B_t)(B_T-B_t)\right]\left(\frac{1}{T-t}-\frac{1}{T-\ut}\right)  \notag \\
  &\qquad\quad + \frac{1}{T-\ut}\left(\e\left[g(x+B_T-B_t)(B_T-B_t)\right]-\e\left[g(x+B^n_T-B^n_t)(B^n_T-B^n_t)\right] \right).
\end{align}
For the first term, since
\begin{equation*}
	\e\left[g(x+B_T-B_t)(B_T-B_t)\right] = \e\left[\left(g(x+B_T-B_t)-g(x)\right)(B_T-B_t)\right],
\end{equation*}
we have, using the fact that $g$ is $\ep$-Hölder continuous,
\begin{align*}
\left|\e\left[g(x+B_T-B_t)(B_T-B_t)\right]\left(\frac{1}{T-t}-\frac{1}{T-\ut}\right)\right| & \leq \|g\|_{\ep}\, \e\left[|B_T-B_t|^{1+\ep}\right]\frac{t-\ut}{(T-t)(T-\ut)} \\
& \le \|g\|_{\ep}\, \frac{t-\ut}{(T-t)^{\frac{1-\ep}{2}}(T-\ut)}.
\end{align*}
But,  exploiting  that $(t-\ut)^{{1-\frac{\ep}{2}}}\le (T-\ut)^{{1-\frac{\ep}{2}}}$ and
$\frac{1}{(T-\ut)^{\frac{\ep}{2}}}\le \frac{1}{(T-t)^{\frac{\ep}{2}}}$, we obtain
\begin{equation*}
	\frac{t-\ut}{(T-t)^{\frac{1-\ep}{2}}(T-\ut)} = \frac{(t-\ut)^{\frac{\ep}{2}}(t-\ut)^{1-\frac{\ep}{2}}}{(T-t)^{\frac{1-\ep}{2}}(T-\ut)^{1-\frac{\ep}{2}}(T-\ut)^{\frac{\ep}{2}}} \leq \frac{(t-\ut)^{\frac{\ep}{2}}}{\sqrt{T-t}} \leq \frac{h^{\frac{\ep}{2}}}{\sqrt{T-t}},
\end{equation*}
from which we deduce that
\begin{equation*}
	\left |\e\left[g(x+B_T-B_t)(B_T-B_t)\right]\left(\frac{1}{T-t}-\frac{1}{T-\ut}\right) \right |  \leq T^{\frac{\ep}{2}}\,\|g\|_{\ep}\frac{n^{-\frac{\ep}{2}}}{\sqrt{T-t}}.
\end{equation*}
Since $\delta(t,T)=T-\ut$, from Corollary~\ref{en:Hrate}, the absolute value of the second term on the RHS of \eqref{g-difference} is bounded by 
\begin{equation*}
	C\,\|g\|_{\ep}\,n^{-\frac{\ep}{2}}\,\frac{(T-\ut)^{1/2}}{T-\ut} = C\,\|g\|_{\ep}\, n^{-\frac{\ep}{2}}\,\frac{1}{\sqrt{T-\ut}} \leq C\,\|g\|_{\ep}\,n^{-\frac{\ep}{2}}\,\frac{1}{\sqrt{T-t}}.
\end{equation*}

Then we get
\begin{equation}\label{eq:dg}
\left| \e\left[ g\left(B^{t,x}_T\right)\frac{B^{t,x}_T -
  x}{T-t}\right]-\e\left[g\left(B^{n,t,x}_T\right)\frac{B^{n,t,x}_T-x}{T-\ut}\right] \right| 
\leq C\,\|g\|_{\ep}\,\frac{n^{-\frac{\ep}{2}}}{\sqrt{T-t}} .
\end{equation}

% paragraph study_of_the_g_difference (end)

\paragraph{Study of the $f$ difference} % (fold)
%\label{par:study_of_the_f_difference}
%
 Here we have  to estimate for $t \in [0,T[$,
\begin{equation*}
	H(t)= \left|\int_t^T \e\left[F\left(s,B^{t,x}_s\right)\frac{B^{t,x}_s-x}{s-t}\right] ds - \int_{\ut+h}^T \e\left[f\left(\os,\Theta^{n,t,x}_s\right)\frac{B^{n,t,x}_s-x}{\us-\ut}\right] ds\right|.
\end{equation*}

When  $T-h\leq t<s<T,$  we observe that $B^{n,t,x}_s-x=0,$  and combining the regularity~\eqref{eq:regF} of $F$ with the estimate \eqref{holder-estimate} we obtain 
\begin{align*}
	H(t) \leq     \int_t^T \e  \left|F\left(s,B^{t,x}_s\right)\frac{B^{t,x}_s-x}{s-t}   \right|  ds  & \leq C\, \int_t^T \frac{ds}{(s-t)^{(1-\ep)/2}\sqrt{T-s}}\, ds = C (T-t)^{\frac{\ep}{2}}\, B\big(\tfrac{1+\ep}{2},\tfrac{1}{2} \big).
\end{align*}
Thus, for $T-h\leq t\leq T$, $H(t)\leq C\, n^{-\frac{\ep}{2}}$.

Let us now consider the case where $0\leq t<T-h$ i.e. $\ut+h \leq T-h$. We first write
\begin{align*}
	\int_t^T \e\left[F\left(s,B^{t,x}_s\right)\frac{B^{t,x}_s-x}{s-t}\right] ds & = \int_{\ut+h}^T \e\left[F\left(s,B^{t,x}_s\right)\frac{B^{t,x}_s-x}{s-t}\right] ds \\
       & \quad + \int_t^{\ut+h} \e\left[F\left(s,B^{t,x}_s\right)\frac{B^{t,x}_s-x}{s-t}\right] ds.
\end{align*}
For the second term of the RHS of this equality, we proceed as above and get 
\begin{align*}
	\left| \int_t^{\ut+h} \e\left[F\left(s,B^{t,x}_s\right)\frac{B^{t,x}_s-x}{s-t}\right] ds \right| 
	& \leq C\, \int_t^{\ut+h} \frac{ds}{(s-t)^{(1-\ep)/2}\sqrt{T-s}}\, ds. 
\end{align*}
But, since $\ut+h\leq T-h$, for $t<s<\ut+h$, $\dfrac{1}{\sqrt{T-s}} \leq \dfrac{1}{\sqrt{h}}$ and, since $0 < \ut+h - t\leq h$,
\begin{equation*}
	\int_t^{\ut+h} \frac{ds}{(s-t)^{(1-\ep)/2}\sqrt{T-s}}\, ds \leq \dfrac{1}{\sqrt{h}}\, \int_t^{\ut+h}\frac{ds}{(s-t)^{(1-\ep)/2}} = \frac{2}{1+\ep}\, \frac{(\ut+h-t)^{(1+\ep)/2}}{\sqrt{h}} \leq \frac{2 h^{\frac{\ep}{2}}}{1+\ep}.
\end{equation*}

Secondly, we split the term
\begin{equation*}
	\int_{\ut+h}^T \e\left[f\left(\os,\Theta^{n,t,x}_s\right)\frac{B^{n,t,x}_s-x}{\us-\ut}\right] ds
\end{equation*}
into two parts:
\begin{equation*}
	\int_{\ut+h}^T \e\left[f\left(s,\Theta^{n,t,x}_s\right)\frac{B^{n,t,x}_s-x}{\us-\ut}\right] ds = \int_{\ut+h}^T \e\left[F_n\left(s,B^{n,t,x}_s\right)\frac{B^{n,t,x}_s-x}{\us-\ut}\right]ds
\end{equation*}
and, the remaining term
\begin{equation*}
	R=\int_{\ut+h}^T \e\left[\left(f\left(\os,\Theta^{n,t,x}_s\right)-f\left(s,\Theta^{n,t,x}_s\right)\right)\frac{B^{n,t,x}_s-x}{\us-\ut}\right] ds.
\end{equation*}
But, due to the uniform regularity of $f$ in time, we have, since $\us-\ut = (\us+h)-(\ut+h) \geq s-(\ut+h)$,
\begin{align*}
	|R| & \leq \|f_t\|_\alpha\, \int_{\ut+h}^T (\os-s)^{\alpha} \frac{\e\left[|B^{n,t,x}_s-x|\right]}{\us-\ut}\, ds \leq \|f_t\|_\alpha\, h^{\alpha}\, \int_{\ut+h}^T \frac{ds}{\sqrt{\us-\ut}}\, ds \\
	& \leq \|f_t\|_\alpha\, h^{\alpha}\, \int_{\ut+h}^T \frac{ds}{\sqrt{s-(\ut+h})}\, ds \leq C\, h^{\alpha}.
\end{align*}

Thus, for $0\leq t <T-h$,
\begin{equation*}
	H(t) \leq \int_{\ut+h}^T \left|\e\left[F\left(s,B^{t,x}_s\right)\frac{B^{t,x}_s-x}{s-t}\right] - \e\left[F_n\left(s,B^{n,t,x}_s\right)\frac{B^{n,t,x}_s-x}{\us-\ut}\right] \right| ds + C\, n^{-(\alpha\wedge \frac{\ep}{2})}.
\end{equation*}

We split the integrand of the first term on the RHS of the inequality into three parts,
\begin{align*}
	& \left|\e\left[F\left(s,B^{t,x}_s\right)\frac{B^{t,x}_s-x}{s-t}\right] - \e\left[F_n\left(s,B^{n,t,x}_s\right)\frac{B^{n,t,x}_s-x}{\us-\ut}\right] \right| \\
	& \qquad \leq \left|\e\left[F\left(s,B^{t,x}_s\right)\frac{B^{t,x}_s-x}{s-t}\right] - \e\left[F\left(s,B^{t,x}_s\right)\frac{B^{t,x}_s-x}{\us-\ut}\right] \right| \\
	& \qquad\quad + \left|\e\left[F\left(s,B^{t,x}_s\right)\frac{B^{t,x}_s-x}{\us-\ut}\right] - \e\left[F\left(s,B^{n,t,x}_s\right)\frac{B^{n,t,x}_s-x}{\us-\ut}\right] \right| \\
	& \qquad\quad + \left|\e\left[F\left(s,B^{n,t,x}_s\right)\frac{B^{n,t,x}_s-x}{\us-\ut}\right] - \e\left[F_n\left(s,B^{n,t,x}_s\right)\frac{B^{n,t,x}_s-x}{\us-\ut}\right] \right| \\
	& \qquad =: H_1(s) + H_2(s) + H_3(s),
\end{align*}
so that
\begin{equation*}
	H(t) \leq \int_{\ut+h}^T \left(H_1(s)+H_2(s)+H_3(s)\right) ds + C\, n^{-(\alpha\wedge \frac{\ep}{2})}.
\end{equation*}

\paragraph{The term $H_1$} % (fold)
%\label{par:the_term_h_1}
%
Since $B^{t,x}_s-x$ has mean zero,
  \begin{align*}
    H_1(s)\le \|F(s,\cdot)\|_{\ep}\e(|B^{t,x}_s-x|^{1+\ep})\left|\frac{1}{s-t}-\frac{1}{\us-\ut}\right|
  \end{align*}
and the regularity~\eqref{eq:regF} of $F$ gives 
\begin{equation*}
	H_1(s) \leq C\, \frac{(s-t)^{(1+\ep)/2}}{\sqrt{T-s}}\, \frac{(t-\ut)+(s-\us)}{(s-t)(\us-\ut)}
	=  \frac{C}{\sqrt{T-s}}\, \frac{(t-\ut)+(s-\us)}{(s-t)^{(1-\ep)/2}(\us-\ut)}.
\end{equation*}
Since $(t-\ut)=(t-\ut)^{\frac{\ep}{2}}(t-\ut)^{1-\frac{\ep}{2}} \leq h^{\frac{\ep}{2}}(\us-\ut)^{1-\frac{\ep}{2}}$ and the same upper bound holds for $s-\us$ (since $s-\us\leq h \leq \us-\ut$), we get
\begin{equation*}
	H_1(s) \leq \frac{C}{\sqrt{T-s}}\, \frac{h^{\frac{\ep}{2}}}{(s-t)^{(1-\ep)/2}(\us-\ut)^{\frac{\ep}{2}}}.
\end{equation*}
Finally, we remark that $\us-\ut = (\us+h)-(\ut+h) \geq s-(\ut+h)$ and $s-t \geq s-(\ut+h)$, to obtain
\begin{equation*}
	H_1(s) \leq \frac{C}{\sqrt{T-s}}\, \frac{h^{\frac{\ep}{2}}}{(s-(\ut+h))^{(1-\ep)/2}(s-(\ut+h))^{\frac{\ep}{2}}} \leq \frac{C\, h^{\frac{\ep}{2}}}{\sqrt{T-s}\, \sqrt{s-(\ut+h)}},
\end{equation*}
and, as a consequence,
\begin{equation*}
	\int_{\ut+h}^T H_1(s)\, ds \leq C\, h^{\frac{\ep}{2}}\, B\big (\tfrac{1}{2},\tfrac{1}{2} \big ).
\end{equation*}
%
% paragraph the_term_h_1 (end)

\paragraph{The term $H_2$} % (fold)
%\label{par:the_term_h_2}
%
We use once again the regularity~\eqref{eq:regF} of $F$ together with Corollary~\ref{en:Hrate} to obtain, for $\ut+h<s<T$,
\begin{equation*}
	H_2(s) \leq C\, n^{-\frac{\ep}{2}}\, \delta(t,s)^{1/2} \, \frac{1}{\sqrt{T-s}}\, \frac{1}{\us-\ut}.
\end{equation*}
We first use the fact that $\delta(s,t)=\max(s-t,\us-\ut)\le s-\ut \leq s-\us + \us-\ut \leq 2(\us-\ut)$ to get 
\begin{equation*}
	H_2(s) \leq C\, \frac{n^{-\frac{\ep}{2}}}{\sqrt{T-s}\sqrt{\us-\ut}},
\end{equation*}
and, since $\us-\ut \geq s-(\ut+h)$, we have
\begin{equation*}
	H_2(s) \leq C\, \frac{n^{-\frac{\ep}{2}}}{\sqrt{T-s}\sqrt{s-(\ut+h)}},
\end{equation*}
from which we deduce the estimate
\begin{equation*}
	\int_{\ut+h}^T H_2(s)\, ds \leq C\, n^{-\frac{\ep}{2}}\, B(\tfrac{1}{2},\tfrac{1}{2}).
\end{equation*}
%
% paragraph the_term_h_2 (end)

\paragraph{The term $H_3$} % (fold)
%\label{par:the_term_h_3}
%
For this last term, we come back to the definitions~\eqref{eq:defF} and \eqref{eq:defFn} of $F$ and $F_n,$ respectively.  By \eqref{FminusF-n} we have, for $\ut+h<s<T$,
\begin{align*}
	H_3(s) & \leq \e\left[|F-F_n|(s,B^{n,t,x}_s) \, \left| \frac{B^{n,t,x}_s-x}{\us-\ut} \right|\right] \\
	&\leq  \|f\|_\lip \, \left(\beta_n(s)+\gamma_n(s)\right) \e\left[\left(1+|B^{n,t,x}_s|\right)^\ep \left| \frac{B^{n,t,x}_s-x}{\us-\ut} \right| \right], 
\end{align*} 
and, by the Cauchy-Schwarz inequality, we derive the estimate
\begin{equation*}
	H_3(s) \leq C\, \left(\beta_n(s)+\gamma_n(s)\right) (1+|x|)^{\ep}\, \frac{1}{\sqrt{\us-\ut}}.
\end{equation*}

% paragraph the_term_h_3 (end)

\paragraph{Summary for $H$} % (fold)
%\label{par:summary_for_h}
%
Let us summarize the estimates we got for $H$. For $T-h\leq t \leq T$, we have $H(t) \leq C\, n^{-\frac{\ep}{2}}$. For $0\leq t < T-h$ we obtained the upper bound
\begin{align*}
	H(t) & \leq C\, n^{-(\alpha\wedge \frac{\ep}{2})}  
	 + C\, (1+|x|)^\ep\, \int_{\ut+h}^T \left(\beta_n(s)+\gamma_n(s)\right)\, \frac{ds}{\sqrt{\us-\ut}}.
\end{align*}
Hence we have, for $t \in [0,T]$,
\begin{equation*}
	H(t) \leq C\, (1+|x|)^\ep \left( n^{-(\alpha\wedge \frac{\ep}{2})}  +  \int_{(\ut+h) \wedge T}^T \left(\beta_n(s)+\gamma_n(s)\right)\, \frac{ds}{\sqrt{\us-\ut}}\right).
\end{equation*}

% paragraph summary_for_h (end)
%
% paragraph study_of_the_f_difference (end)

Coming back to~\eqref{eq:dg}, we have, for any $x\in\rset$ and $t \in [0,T[$,
\begin{equation*}
	|\nabla u(t,x)-\Delta^n(t,x)| \\
	\leq C\, (1+|x|)^\ep \left( \frac{n^{-(\alpha\wedge \frac{\ep}{2})}}{\sqrt{T-t}} +  \int_{(\ut+h) \wedge T}^T \left(\beta_n(s)+\gamma_n(s)\right)\, \frac{ds}{\sqrt{\us-\ut}}\right),
\end{equation*}
and, as a byproduct, 
\begin{equation}\label{eq:almostZ}
 \gamma_n(t) \leq C \left( \frac{n^{-(\alpha\wedge \frac{\ep}{2})}}{\sqrt{T-t}} + \int_{(\ut+h) \wedge T}^T \left(\beta_n(s)+\gamma_n(s)\right)\, \frac{ds}{\sqrt{\us-\ut}}\right), \quad t \in [0,T[.
\end{equation}
% subsection estimate_for_nabla_u_delta_n (end)

\subsection*{Global estimate} % (fold)
%\label{sub:global_estimate}
%
Plugging \eqref{eq:almostY} into \eqref{eq:almostZ}, we get, for $t \in [0,T[$,
\begin{equation*}
	\gamma_n(t) \leq C \left( \frac{n^{-(\alpha\wedge \frac{\ep}{2})}}{\sqrt{T-t}} +  \int_{(\ut+h)\wedge T}^T  \left(n^{-(\alpha\wedge\frac{\ep}{2})} + \int_s^T \gamma_n(r)\, dr + \gamma_n(s)\right) \frac{ds}{\sqrt{\us-\ut}} \right).
\end{equation*}
For $t<T-h$, since $\us-\ut = (\us+h)-(\ut+h) \geq s-(\ut+h)$, we have
\begin{equation*}
	n^{-(\alpha\wedge\frac{\ep}{2})}\, \int_{\ut+h}^T
    \frac{ds}{\sqrt{\us-\ut}} \leq n^{-(\alpha\wedge\frac{\ep}{2})}\, \int_{\ut+h}^T \frac{ds}{\sqrt{s-(\ut+h)}} \leq 2\,n^{-(\alpha\wedge\frac{\ep}{2})}\,\sqrt{T-(\ut+h)} \leq 2\sqrt{T}\, n^{-(\alpha\wedge\frac{\ep}{2})}.
\end{equation*}
Again, we have,
\begin{equation*}
	\int_{\ut+h}^T \int_s^T \gamma_n(r)\, dr \frac{ds}{\sqrt{\us-\ut}}
	\le \int_{\ut+h}^T \int_s^T \gamma_n(r)\, dr \frac{ds}{\sqrt{s-(\ut+h)}},
\end{equation*}
from which we deduce
\begin{align*}
	\int_{\ut+h}^T \int_s^T \gamma_n(r)\, dr \frac{ds}{\sqrt{s-(\ut+h)}} & = \int_{\ut+h}^T \gamma_n(r) \int_{\ut+h}^T \ind_{r>s}\frac{ds}{\sqrt{s-(\ut+h)}} dr \\
	& = 2\, \int_{\ut+h}^T \gamma_n(r)\, \sqrt{r-(\ut+h)}\, dr  \\ 
	& \leq 2T\, \int_{\ut+h}^T \gamma_n(s)\, \frac{ds}{\sqrt{\us-\ut}}.
\end{align*}
It follows that, for $t \in [0,T[$,
\begin{equation*}
	\gamma_n(t) \leq C \left( \frac{n^{-(\alpha\wedge \frac{\ep}{2})}}{\sqrt{T-t}} +  \int_{(\ut+h)\wedge T}^T   \gamma_n(s) \frac{ds}{\sqrt{\us-\ut}} \right).
\end{equation*}
Thus,
\begin{align*}
	\gamma_n(t) & \leq C \frac{n^{-(\alpha\wedge \frac{\ep}{2})}}{\sqrt{T-t}}  \\
	& \quad + C^2\, \left(  \int_{(\ut+h)\wedge T}^T  \left(\frac{n^{-(\alpha\wedge \frac{\ep}{2})}}{\sqrt{T-s}} +  \int_{(\us+h)\wedge T}^T  \gamma_n(r)\frac{dr}{\sqrt{\ur-\us}}\right) \frac{ds}{\sqrt{\us-\ut}} \right).
\end{align*}
But we have
\begin{equation*}
	  \int_{(\ut+h)\wedge T}^T  \frac{n^{-(\alpha\wedge \frac{\ep}{2})}}{\sqrt{T-s}}\frac{ds}{\sqrt{\us-\ut}}\le  \int_{(\ut+h)\wedge T}^T  \frac{n^{-(\alpha\wedge \frac{\ep}{2})}}{\sqrt{T-s}}\frac{ds}{\sqrt{s-(\ut+h)}} = n^{-(\alpha\wedge \frac{\ep}{2})}\, B\big (\tfrac{1}{2},\tfrac{1}{2} \big),
\end{equation*}
and, for $t<T-h$, since $\us+h \le r$ if and only if $s<\ur$,
\begin{align*}
	 \int_{(\ut+h)\wedge T}^T  \int_{(\us+h)\wedge T}^T  \gamma_n(r)\frac{dr}{\sqrt{\ur-\us}}\frac{ds}{\sqrt{\us-\ut}}
	&= \int_{(\ut+2h)\wedge T}^T  \gamma_n(r) \int_{(\ut+h)\wedge T}^T  \ind_{s <\ur} \frac{ds}{\sqrt{\ur-\us}\sqrt{\us-\ut}} \, dr \\
	&= \int_{(\ut+2h)\wedge T}^T  \gamma_n(r)  \int_{(\ut+h)\wedge T}^{\ur} \frac{ds}{\sqrt{\ur-\us}\sqrt{\us-\ut}} \, dr.
\end{align*}
But $\ur-\us \geq \ur-s$ and $\us-\ut \geq s-(\ut+h)$, so we get
\begin{align*}
	 \int_{(\ut+h)\wedge T}^T  \int_{(\us+h)\wedge T}^T   \gamma_n(r)\frac{dr}{\sqrt{\ur-\us}}\frac{ds}{\sqrt{\us-\ut}}
	& \leq  \int_{(\ut+2h)\wedge T}^T \gamma_n(r) \int_{(\ut+h)\wedge T}^{\ur} \frac{ds}{\sqrt{\ur-s}\sqrt{s-(\ut+h)}} \, dr \\
	& = \, B\big (\tfrac{1}{2},\tfrac{1}{2} \big)\,  \int_{(\ut+2h)\wedge T}^T  \gamma_n(r)\, dr.
\end{align*}
Finally, we have
\begin{equation*}
	\gamma_n(t) \leq C \left(\frac{n^{-(\alpha\wedge \frac{\ep}{2})}}{\sqrt{T-t}} + \int_t^T \gamma_n(s)\, ds\right) \quad \text{ for } \,\, t \in [0,T[,
\end{equation*}
and from Gronwall's inequality (Lemma \ref{lem:gronwall})
\begin{equation*}
	\gamma_n(t) \leq C\, \frac{n^{-(\alpha\wedge \frac{\ep}{2})}}{\sqrt{T-t}} \quad  \text{ for } \,\, t \in [0,T[.
\end{equation*}

 Coming back to~\eqref{eq:almostY}, we have also, 
\begin{equation*}
	\beta_n(t) \leq C\, n^{-(\alpha\wedge \frac{\ep}{2})}  \quad  \text{ for } \,\, t \in [0,T].
\end{equation*}
%
% subsection global_estimate (end)
%
%
The proof of Proposition~\ref{en:mainPDE} is complete.
\end{proof}

\begin{proof}[Proof of Theorem~\ref{en:mainBSDE}]
Theorem~\ref{en:mainBSDE} is mainly a corollary of Proposition~\ref{en:mainPDE}. 

Let us begin with the convergence of the $(Y^n)_n$ processes. Let us fix $r\geq 1$, $x\in\rset$ and $0\leq t\leq s\leq T$. We have
\begin{align*}
	W_r\left(Y^{t,x}_s,Y^{n,t,x}_s\right)& = W_r\left(u(s,B^{t,x}_s),U^n(s,B^{n,t,x}_s)\right) \\
	&\leq W_r\left(u(s,B^{t,x}_s),u(s,B^{n,t,x}_s)\right) + W_r\left(u(s,B^{n,t,x}_s),U^n(s,B^{n,t,x}_s)\right).
\end{align*}
Since, by Lemma~\ref{en:regu}, $u$ is $\ep$-Hölder continuous in space, uniformly in time, we have, by Hölder's inequality,
\begin{equation*}
	W_r\left(u(s,B^{t,x}_s),u(s,B^{n,t,x}_s)\right) \leq C\, W_r\left(B^{t,x}_s,B^{n,t,x}_s\right)^{\ep} \leq  C c_r^\ep T^{\frac{\ep}{2}}  \, n^{-\frac{\ep}{2}},
\end{equation*}
where we have used Proposition~\ref{en:fromRio} (see \eqref{eq:Wr}) for the last inequality.   Moreover, by  Proposition \ref{en:mainPDE},
\begin{equation*}
	W_r\left(u(s,B^{n,t,x}_s),U^n(s,B^{n,t,x}_s)\right) \leq \left\| u(s,B^{n,t,x}_s)-U^n(s,B^{n,t,x}_s)\right\|_{\lp^r} \leq C\, (1+|x|)^{\ep}\, n^{-(\alpha\wedge \frac{\ep}{2})}. 
\end{equation*} 
This gives the first part of the result. 

Let us continue with the convergence of the $(Z^n)_n$ processes. The proof is almost the same except for the grid points. Let $0\leq t<s\leq T$ with $s\neq \us$ i.e. $s\not\in\{kh, k=n_t+1,\ldots,n\}$. We have as before
\begin{multline*}
        W_r\left(\nabla u(s,B^{t,x}_s),\Delta^n(s,B^{n,t,x}_s)\right) \\
	\leq W_r\left(\nabla u(s,B^{t,x}_s),\nabla u(s,B^{n,t,x}_s)\right) + W_r\left(\nabla u(s,B^{n,t,x}_s),\Delta^n(s,B^{n,t,x}_s)\right).
\end{multline*} 
Since, by Lemma~\ref{en:regnabus}, $\nabla u(s,\cdot)$ is $\ep$-Hölder continuous, we have, by Hölder's inequality,
\begin{equation*}
	W_r\left(\nabla u(s,B^{t,x}_s),\nabla u(s,B^{n,t,x}_s)\right) \leq \frac{C}{\sqrt{T-s}}\, W_r\left(B^{t,x}_s,B^{n,t,x}_s\right)^{\ep} \leq \frac{ C c_r^\ep T^{\frac{\ep}{2}}}{\sqrt{T-s}}\, n^{-\frac{\ep}{2}},
\end{equation*}
where the last inequality follows from~\eqref{eq:Wr}.  Moreover, by Proposition \ref{en:mainPDE},
\begin{equation*}
	W_r\left(\nabla u(s,B^{n,t,x}_s),\Delta^n(s,B^{n,t,x}_s)\right) \leq \left\| \nabla u(s,B^{n,t,x}_s)-\Delta^n(s,B^{n,t,x}_s)\right\|_{\lp^r} \leq C\, \frac{(1+|x|)^{\ep}}{\sqrt{T-s}}n^{-(\alpha\wedge \frac{\ep}{2})},
\end{equation*} 
and the result follows,   in this case, from equalities \eqref{eqZ} and \eqref{eq:RZ}  together with Remark \ref{en:Zreg}. 

Let us now consider the case where $s\in\{kh, k=n_t+1,\ldots,n-1\}.$  In this case we have 
\begin{equation*}
	Z^{n,t,x}_s = \nabla^n U^n\left(s,B^{n,t,x}_{s-}\right)=\Delta^n\left(s-h/2,B^{n,t,x}_{s-h/2}\right) 
\end{equation*}
which is not equal to  $ \Delta^n\left(s,B^{n,t,x}_s\right)$  in general. 
We first write
\begin{align*}
  W_r\left(Z^{t,x}_s,Z^{n,t,x}_s\right) = W_r\left(Z^{t,x}_s,Z^{n,t,x}_{s-h/2}\right) \leq W_r\left(Z^{t,x}_s,Z^{t,x}_{s-h/2}\right) +W_r\left( Z^{t,x}_{s-h/2},Z^{n,t,x}_{s-h/2}\right).
\end{align*}
The second term on  the RHS can be bounded  by using our previous result, namely
\begin{equation*}
	W_r\left(Z^{t,x}_{s-h/2},Z^{n,t,x}_{s-h/2}\right) \leq  C\frac{n^{-(\alpha\wedge \frac{\ep}{2})}\, (1+|x|)^\ep}{\sqrt{T - (s- h/2)}} \leq C\frac{n^{-(\alpha\wedge \frac{\ep}{2})}\, (1+|x|)^\ep}{\sqrt{T - s}}.
\end{equation*}
For the first term, one can write 
\begin{multline*}
		W_r\left(Z^{t,x}_s,Z^{t,x}_{s-h/2}\right)=W_r\left(\nabla u(s,B^{t,x}_s),\nabla u(s-h/2,B^{t,x}_{s-h/2})\right) \\
		\leq W_r\left(\nabla u(s-h/2,B^{t,x}_{s-h/2}),\nabla u(s-h/2,B^{t,x}_{s})\right) + W_r\left(\nabla u(s-h/2,B^{t,x}_{s}),\nabla u(s,B^{t,x}_s)\right).
\end{multline*}
From Lemma~\ref{en:regnabus}, $\nabla u(s-h/2,\cdot)$ is $\ep$-Hölder continuous and we have
\begin{align*}
	W_r\left(\nabla u(s-h/2,B^{t,x}_{s-h/2}),\nabla u(s-h/2,B^{t,x}_{s})\right) & \leq \frac{C}{\sqrt{T-(s-h/2)}}\, W_r\left(B^{t,x}_{s-h/2},B^{t,x}_{s}\right)^{\ep} \\
	& \leq \frac{C \|G\|_{L^r}^\ep}{\sqrt{T-s}} \, \left (\frac{h}{2} \right )^{\frac{\ep}{2}}\!\!,
\end{align*}
where $G$ is a standard normal random variable.
Finally, for the last term, we  use Proposition~\ref{en:timenabu}  for the time regularity of $\nabla u$. We have
\begin{equation*}
	W_r\left(\nabla u(s-h/2,B^{t,x}_{s}),\nabla u(s,B^{t,x}_s)\right) \leq \left\|\nabla u(s-h/2,B^{t,x}_{s})-\nabla u(s,B^{t,x}_s) \right\|_{\lp^r} \leq \frac{C \, h^{\frac{\ep}{2}}}{\sqrt{T-s}}, 
\end{equation*}
and the estimate for $W_r\left(Z^{n,t,x}_s,Z^{t,x}_s\right)$ follows.

This ends the proof.  
\end{proof}
%
% section main_results (end)
 
\appendix
\section{Appendix}
\subsection{A priori estimate for discrete BSDEs} % (fold)
\label{sub:a_priori_estimate_for_discrete_bsdes}

For the  convenience  of the reader, we prove a generalization of an a priori estimate for BSDEs driven by random walks given in  \cite[Proposition 7]{BDM02} (see also the appendix in \cite{BDM01}). This generalization allows to consider two different generators.

\begin{lemme}\label{en:aeDB}
	There exists an integer $n_0   \in \nset^*$ and a constant  $C>0$ both depending only on $T$ and $\|f\|_\lip$ such that for any couple of functions $(\bar g, \bar f)$ satisfying \ref{H:A1} and  for all $n\geq n_0 $ with $n> T\| \bar f\|_\lip $  and all    $(t,x,\bar x)\in[0,T]\times\rset^2$,
	\begin{multline*}
		\e\left[\sup_{t\leq s\leq T}\left|Y^{n,t,x}_s-\bar Y^{n,t,\bar x}_s\right|^2 +\int_{]t,T]} \left|Z^{n,t,x}_r-\bar Z^{n,t,\bar x}_r\right|^2 d\langle B^n\rangle_r\right]\\ 
		\leq C\, \e\left[\left|g\left(B^{n,t,x}_T\right)-\bar g\left(B^{n,t,\bar x}_T\right)\right|^2+ \int_{]t,T]} \left|\delta f\left(r,x,\bar x, \bar Y^{n,t,\bar x}_{r^-}, \bar Z^{n,t,\bar x}_r\right) \right|^2 d\langle B^n\rangle_r \right],
	\end{multline*}
	where $\delta f(r,x,\bar x,y,z)= f\left(r,B^{n,t,x}_{r^-},y,z\right)-\bar
    f\left(r,B^{n,t,\bar x}_{r^-},y,z\right)$ and, for all $(t,x)$, $\left(\bar Y^{n,t,x},\bar Z^{n,t,x}\right)$ denotes the solution to \eqref{eq:mainBSDEn} where $(g,f)$ is replaced by $(\bar g,\bar f)$.
\end{lemme}

\begin{proof}
	Let $n$ be such that  $T\|f\|_\lip/n < 1$    and $T\| \bar f\|_\lip/n < 1$.   Since, $\langle
    B^n\rangle_t - \langle B^n\rangle_s \leq (t-s) + T/n$, doing exactly
    the same computation as in the proof of  Proposition~7 in \cite{BDM02}, we get, for a universal constant $c\geq 1$,
	\begin{multline*}
	\e\left[\sup_{ \sigma \leq s \leq\tau}|\delta Y^n_s|^2 +\int_{]\sigma,\tau]} |\delta Z^n_r|^2 d\langle B^n\rangle_r\right]  \\
	\leq c\,\e\left[|\delta Y^n_\tau|^2 \right]  + C(\tau-\sigma,n)\, \e\left[\int_{]\sigma,\tau]} \left|\delta f\left(r,x,\bar x,\bar Y^{n,t,\bar x}_{r^-}, \bar Z^{n,t,\bar x}_r\right) \right|^2 d\langle B^n\rangle_r \right]  \\
	+ C(\tau-\sigma,n)\, \e\left[\sup_{ \sigma \leq s \leq\tau} |\delta Y^n_s|^2 + \int_{]\sigma,\tau]} |\delta Z^n_r|^2\,d\langle B^n\rangle_r \right]
	\end{multline*}
	 for all deterministic $0\le \sigma <\tau \le T, $ where 
	\begin{equation*}
			C(\alpha,n)= c\, \max\left(\|f\|_\text{Lip}^2,1\right)\, \max\left( \left(\alpha+T/n\right)^2,\left(\alpha+T/n\right)\right)  \quad \text{ for } \alpha \ge 0.
	\end{equation*}
	We choose an integer $m  \in \nset^*$ such that, with $\alpha=T/m$, $C(\alpha,+\infty)\leq 1/3$. Then, there exists an $n_0 \in \nset^*$ such that, for $n\geq n_0$  it holds $C(\alpha,n)\leq 1/2$   and $T\|f\|_\lip/n < 1.$  As soon as $\tau-\sigma \leq \alpha$ and $n\geq n_0$, we have
	\begin{multline} \label{delta-a-priori}
		\e\left[\sup_{\sigma \leq s\leq\tau}|\delta Y^n_s|^2 +\int_{]\sigma,\tau]} |\delta Z^n_r|^2 d\langle B^n\rangle_r\right]  \\
		\leq 2c\,\e\left[|\delta Y^n_\tau|^2 + \int_{]\sigma,\tau]} \left|\delta f\left(r,x,\bar x,\bar Y^{n,t,\bar x}_{r^-}, \bar Z^{n,t,\bar x}_r\right) \right|^2 d\langle B^n\rangle_r \right].
	\end{multline} 
	We set, for $0\leq k\leq m-1$, $I_k=\ut+ ] k(T-\ut)/m,(k+1)(T-\ut)/m]$ and we introduce the following norm on $\mathcal{S}^2\times\mathcal{M}^2$:
	\begin{equation*}
		\left\| \left(Y^n,Z^n\right)\right\|_s^2 =\sum_{k=0}^{m-1} (2 \times 2c)^k\, \e\left[\sup\nl_{t\in I_k} |Y^n_t|^2 + \int_{I_k}|Z^n_r|^2 d\langle B^n\rangle_r \right] .
	\end{equation*}
	 Considering $I_k  =]\sigma,\tau]$ and summing up  \eqref{delta-a-priori} over $k$ yields 
	\begin{multline*}
		 \left\| \left(\delta Y^n,\delta Z^n\right)\right\|_s^2 \leq \frac{1}{2}\, \left\| \left(\delta Y^n,\delta Z^n\right)\right\|_s^2 + (4c)^{m-1} (2c)\, \e\left[\left|g\left(B^{n,t,x}_T\right) - \bar g\left(B^{n,t,\bar x}_T\right) \right|^2\right]  \\ + (4c)^{m-1}\,  (2c)\,\e\left[\int_{]\ut,T]} \left|\delta f\left(r,x,\bar x,\bar Y^{n,\ut,x}_{r^-}, \bar Z^{n,\ut,x}_r\right) \right|^2 d\langle B^n\rangle_r \right],
	\end{multline*}
	from which we get 
	\begin{equation*}
		\left\| \left(\delta Y^n,\delta Z^n\right)\right\|_s^2  \leq  (4c)^{m}\, \e\left[\left|g\left(B^{n,t,x}_T\right) - \bar g\left(B^{n,t,\bar x}_T\right) \right|^2 + \int_{]\ut,T]} \left|\delta f\left(r,x,\bar x,\bar Y^{n,t,\bar x}_{r^-}, \bar Z^{n,t,\bar x}_r\right) \right|^2 d\langle B^n\rangle_r \right].
	\end{equation*}
This finishes the proof since $ \left\| \left(\delta Y^n,\delta Z^n\right)\right\|^2_s$  upper bounds  the LHS of the inequality  stated in the Lemma.
\end{proof}

% subsection a_priori_estimate_for_discrete_bsdes (end)
\newpage
\subsection{Gronwall lemmas}
 
 We recall  the  Gronwall lemmas used in this article. 

\begin{lemme} \label{lem:gronwall}
 Suppose that $g(t), \alpha(t):[0,T[ \to [0,\infty[$ are integrable functions, and $\beta > 0.$
For $0 \le t <T ,$ if
$$ g(t) \le \alpha(t) + \beta\int_t^T   g(s)ds, $$
then
$$ g(t) \le \alpha(t) + \beta\int_t^T   \alpha(s)    e^{ \beta (s-t)}ds. $$
\end{lemme} 
\bigskip

The second lemma is of Volterra type. It can be either proved directly by a convolution argument or one can use 
\cite[Exercise 4, page 190]{henry}.

\begin{lemme}
\label{volterra_gronwall}
Assume a measurable $g:[0,T[\to [0,\infty[\in L_1([0,T[)$  and $\alpha,\beta>0$ such that 
\begin{align*}
g(t) \le \frac{\alpha}{\sqrt{T-t}} + \beta \int_t^T \frac{g(s)}{\sqrt{s-t}} ds 
\end{align*}
for all $t\in [0,T[$. Then $g(t) \le C \frac{\alpha}{\sqrt{T-t}}$ for $t\in [0,T[$ for a constant
$C=C(T,\beta)>0$.
\end{lemme}

\subsection{Proof of Lemma \ref{en:regnabus}: Step 2}
\label{A: generator Hoelder}
We assume that  $0<\ep<1$; the case $\ep=1$ was treated in Step 1. For $\eta>0$, let us consider the function
\begin{align*}
	f^{\eta}(t,x,y,z) & = \inf\{ f(t,p,y,z)+\eta |x-p| : p\in\rset\} = \inf\{f(t,x-q,y,z)+\eta |q| : q\in\rset\}.
\end{align*}
When $f$ satisfies \eqref{eq:prop_f},  then  $f^\eta$ does as well and it is $\eta$-Lipschitz continuous w.r.t.~$x$. Moreover
 \begin{equation}
	 \label{eq:unifmaj}
 	\left| f^{\eta}(t,x,y,z)-f(t,x,y,z) \right| \leq c(\|f_x\|_\ep) \, \eta^{-\ep/(1-\ep)},
 \end{equation}
 with $c(t) = (\ep t)^{\ep/(1-\ep)} t(1-\ep)$. Indeed,
 \begin{align*}
 	f(t,x,y,z)  \geq f^\eta(t,x,y,z) & =  f(t,x,y,z) + \inf\{f(t,x-q,y,z)-f(t,x,y,z)+\eta |q| : q\in\rset\} \\
	& \geq f(t,x,y,z) + \inf\{ - \|f_x\|_\ep |q|^\ep +\eta |q| : q\in\rset\} \\
	& = f(t,x,y,z) - c(\|f_x\|_\ep)\, \eta^{-\ep/(1-\ep)}.
 \end{align*}
In particular, $|K_f - K_{f^\eta}|\le  c(\|f_x\|_\ep)\, \eta^{-\ep/(1-\ep)}$.
Let $(Y^{t,x,\eta},Z^{t,x,\eta})$ be the solution to the BSDE~\eqref{eq:mainBSDE} with
data $(g, f^\eta)$ and $u^\eta$ be the function $u^{\eta}(t,x) :=Y^{t,x,\eta}_t$. By the usual classical estimate for BSDEs
(see, for instance,  \cite[Proposition 2.1 and remarks]{elkaroui}  or \cite[Lemma 5.26]{GY}), 
there exists a $C>0$ such that, for all $(t,x)\in[0,T]\times\rset$,
\begin{align} \label{apriori}
\e \left[ \sup_{s \in [t,T]} |Y^{t,x,\eta} _s  - Y^{t,x} _s |^2  +   \int_t^T  |Z^{t,x,\eta} _s  - Z^{t,x} _s |^2ds  \right ]  \notag
& \le   C \e\int_t^T |(f^\eta -f)(s,B^{t,x}_s, Y^{t,x}_s, Z^{t,x}_s)|^2ds \notag \\
& \le   C\, T\, c(\|f_x\|_\ep)^2\, \eta^{-2\ep/(1-\ep)}.
 \end{align}
 In particular, $(u^\eta)_\eta$ converges to $u$, as $\eta\to+\infty$, uniformly on $[0,T]\times\rset$.
 
%% Now the proof of Item  \eqref{u-hoelder-time} is a consequence of the uniform convergence from $(u^\eta)_\eta$ to $u$.

Proof of \eqref{ma_zhang_formula}, \eqref{nabla u-diff}, and \eqref{nabla u-bound}.
Since $f^\eta$ is Lipschitz continuous w.r.t. $(x,y,z)$ and satisfies \eqref{eq:prop_f} (uniformly in $\eta$), by Step 1, we know that $u^\eta\in \m C^{0,1}([0,T[\times\rset)$  and $Z^{t,x,\eta} _s = \nabla u^\eta(s,B^{t,x}_s)$ for a.e.~$(s,\omega)  \in [t,T[\times \Omega $  with
\begin{align} \label{nabla-u-bound}
 |\nabla u^\eta (s,x)|\le\frac{C}{(T-s)^{(1-\ep)/2}} \quad \text{ for } \,\,(s,x)  \in [0,T[\times \rset.
\end{align}
Taking into account the convergence of $(Z^{t,x,\eta})_\eta$ to $Z^{t,x}$ in $\lp^2( [t,T[\times\Omega)$, we have
\begin{align} \label{z-bound}
|Z^{t,x} _s|\le\frac{C}{(T-s)^{(1-\ep)/2}}, \,\, \text{  for    a.e.}  \,\, (s,\omega)\in [t,T[\times \Omega. 
\end{align}   
We define  the function
$$ v(t,x) := \e\left(g(B^{t,x}_T)\frac{B_T-B_t}{T-t} + \int_t^T
                     f(s,B^{t,x}_s, Y^{t,x}_s, Z^{t,x}_s)\frac{B_s-B_t}{s-t}ds\right)  \,\, \text{ for } \,(t,x)  \in [0,T[\times \rset .$$
First we show that 
$$ \nabla u^\eta (t,x) \to v(t,x) \quad \mbox{ for } \quad (t,x) \in [0,T[\times \rset $$
which also implies that $v:[0,T[\times \rset \to \rset$ is measurable.
For this we denote
\begin{align*}
 \widehat Z_r^{t,x}:=\e_r \left( g(B^{t,x}_T)\frac{B_T-B_r}{T-r}+ \int_r^T
                     f(s,B^{t,x}_s, Y^{t,x}_s, Z^{t,x}_s)\frac{B_s-B_r}{s-r}ds \right) \quad \text{ for } r\in [t,T[ . 
\end{align*}  
We also use \eqref{eq: nabla-u-b} for $(g, f^\eta)$ 
so that 
$$   \nabla u^\eta (r,B_r^{t,x})  
   = \e_r\left(g(B^{t,x}_T)\frac{B_T-B_r}{T-r} + \int_r^T
                     f^\eta(s,B^{t,x}_s, Y^{t,x,\eta}_s, Z^{t,x,\eta}_s)\frac{B_s-B_r}{s-r}ds\right). $$
Then  we apply the   conditional  Cauchy-Schwarz inequality  to  $\tfrac{(f^\eta -f)}{(s-r)^{\ep/4}}  \tfrac{B_s-B_r}{(s-r)^{1-\ep/4}} $   
and  use \eqref{eq:unifmaj} to get 
\begin{align*}
& \e| \nabla u^\eta (r,B_r^{t,x})-  \widehat Z_r^{t,x} |^2 \\
& \le \e \left | \e_r\left(\int _r^T(f^\eta (s,B^{t,x}_s, Y^{t,x,\eta}_s, Z^{t,x,\eta}_s)   
-    f(s,B^{t,x}_s, Y^{t,x}_s, Z^{t,x}_s))\frac{B_s-B_r}{s-r}ds\right) \right |^2 \\
& \le \frac{2}{\ep} T^{\frac{\ep}{2}} \, \e \int _r^T \Big( c(\|f_x\|_\ep)\,\eta^{-\ep/(1-\ep)}  +\|f_y\|_{\lip}| Y^{t,x,\eta}_s-Y^{t,x}_s| + \|f_z\|_{\lip} |Z^{t,x,\eta}_s-Z^{t,x}_s| \Big)^2 
\frac{1}{(s-r)^{\ep/2}}ds. 
 \end{align*}   
 Taking into account the bound for $Y$ given in~\eqref{apriori}, we have
 \begin{align}
	 \label{eq:forgronwall}
 \e| \nabla u^\eta (r,B_r^{t,x})-  \widehat Z_r^{t,x} |^2 & \leq C \left( \eta^{-2\ep/(1-\ep)} + 
  \e \int _r^T  |Z^{t,x,\eta}_s-Z^{t,x}_s|^2 \frac{1}{(s-r)^{\ep/2}}\, ds\right),
  \end{align}
  where $C$ depends on $T$, $\ep$, $\|f_x\|_\ep$ and $\|f\|_\lip$.
  
Now  we find a sequence $\eta_k \uparrow +\infty$ such that the RHS tends to $0$. Indeed, this follows by dominated convergence as  \eqref{apriori} guarantees a sequence $\eta_k \uparrow +\infty$  such that 
$$
Z^{t,x,\eta_k}_s \to   Z^{t,x}_s \text{  a.e. \,\, on \, }  [t,T[\times \Omega,
$$
and because of the equations  \eqref{nabla-u-bound} and \eqref{z-bound}. For $r\geq t$, $\nabla u^\eta (r,B_r^{t,x})$ converges to $\widehat Z_r^{t,x}$ in $\lp^2$ and, in particular, for $r=t$ we obtain the desired convergence 
$\nabla u^\eta (t,x) \to v(t,x)$. Because  of $Z^{t,x,\eta} _s = \nabla u^\eta(s,B^{t,x}_s)$ for a.e.~$(s,\omega)  \in [t,T[\times \Omega $ this also gives that
\begin{align} \label{Z-by-v}
Z_s^{t,x} = v(s,B_s^{t,x})= \widehat Z_s^{t,x} \text{  a.e. \,\, on \, }  [t,T[\times \Omega.
 \end{align}
Coming back to~\eqref{eq:forgronwall},  Gronwall's lemma  (Lemma \ref{volterra_gronwall}) gives,
\begin{align*}
\e| Z_r^{t,x,\eta}- Z_r^{t,x}  |^2  \le C\eta^{-2\ep/(1-\ep)} .
\end{align*}  
Especially, $  |\nabla u^\eta(t,x) - v(t,x)|  \le C\eta^{-\ep/(1-\ep)}$,  so that $\lim_{\eta\to\infty } \nabla u^\eta (t,x) = v(t,x)$.
 
Moreover, we conclude from this estimate and from  the continuity of  $ \nabla u^\eta$  on $[0,T[\times \rset$   that also  $v$ is continuous.   
Finally, $ v(t,x)=\nabla u(t,x)$ follows  from taking the limit $\eta \uparrow +\infty$  in 
$$  u^\eta (t,x) = u^\eta (t,0) + \int_0^x  \nabla u^\eta (t,y) dy, $$
where we use dominated convergence based on the inequality \eqref{nabla-u-bound}. 

%:	References
%
% \bibliography{ref}
% \bibliographystyle{amsplain}

\providecommand{\bysame}{\leavevmode\hbox to3em{\hrulefill}\thinspace}
\providecommand{\MR}{\relax\ifhmode\unskip\space\fi MR }
\providecommand{\MRhref}[2]{%
  \href{http://www.ams.org/mathscinet-getitem?mr=#1}{#2}
}
\providecommand{\href}[2]{#2}

\end{document}